\documentclass[12pt]{amsart}
\usepackage{amssymb,amsmath,amsthm,amssymb}
\usepackage{mathrsfs}  
\usepackage{mathtools}
\usepackage{graphicx}
\usepackage{comment}
\usepackage{hyperref}
\usepackage{enumerate}
\usepackage{color}
\usepackage{epstopdf}
\renewcommand{\epsilon}{\varepsilon}

\DeclareMathOperator{\proj}{proj}

\DeclareMathOperator{\graph}{graph}

\DeclareMathOperator{\relint}{rel\, int}
\def\R{\mathbb{R}}

\def\N{\mathbb{N}}
\def\Z{\mathbb{Z}}

\def\d{\delta}
\def\a{\alpha}
\def\e{\epsilon}

\def\r{\rho}

\def\th{\theta}

\def\s{\sigma}
\def\o{\omega}
\def\l{\lambda}

\def\M{\mathcal{M}}

\def\ov{\overline}

\newcommand{\inner}[2]{\left\langle#1,#2\right\rangle} 

\newtheorem{theorem}{Theorem}[section]
\newtheorem{lemma}[theorem]{Lemma}
\newtheorem{proposition}[theorem]{Proposition}

\newtheorem{remark}[theorem]{Remark}
\newtheorem{corollary}[theorem]{Corollary}

\newtheorem{claim}{Claim}[theorem]

\newtheoremstyle{TheoremNum}
        {\topsep}{\topsep}              
        {\itshape}                      
        {}                              
        {\bfseries}                     
        {.}                             
        { }                             
        {\thmname{#1}\thmnote{ \bfseries #3}}
    \theoremstyle{TheoremNum}

\title
{The atomic structure of ancient grain boundaries}
\author{Theodora Bourni}
\author{Mat Langford}
\author{Giuseppe Tinaglia}

\address{Department of Mathematics, University of Tennessee Knoxville, Knoxville TN, 37996-1320}
\email{tbourni@utk.edu, mlangford@utk.edu}

\address{School of Mathematical and Physical Sciences, The University of Newcastle, Newcastle, NSW, Australia, 2308}
\email{mathew.langford@newcastle.edu.au}

\address{Department of Mathematics, King's College London, London, WC2R 2LS, U.K.}
\email{giuseppe.tinaglia@kcl.ac.uk}

\date{\today}

\begin{document}

\begin{abstract}
Democritus and the early atomists held that 
\begin{quotation}
\emph{the material cause of all things that exist is the coming together of atoms and void. Atoms are eternal and have many different shapes, and they can cluster together to create things that are perceivable. Differences in shape, arrangement, and position of atoms produce different phenomena.
} 
\end{quotation}

Like the atoms of Democritus, the Grim Reaper solution to curve shortening flow is eternal and indivisible --- it does not split off a line, and is itself its only ``asymptotic translator''. Confirming the heuristic described by Huisken and Sinestrari [J. Differential Geom. 101, 2 (2015), 267--287], we show that it gives rise to a great diversity of convex ancient and translating solutions to mean curvature flow, through the evolution of families of Grim hyperplanes in suitable configurations. 
We construct, in all dimensions $n\ge 2$, a large family of new examples, including both symmetric and asymmetric examples, as well as many eternal examples that do not evolve by translation. The latter resolve a conjecture of White [J. Amer. Math. Soc. 16, 1 (2003), 123--138].

We also provide a detailed asymptotic analysis of convex ancient solutions in slab regions in general. Roughly speaking, we show that they decompose ``backwards in time'' into a canonical configuration of Grim hyperplanes which satisfies certain necessary conditions. An analogous decomposition holds ``forwards in time'' for eternal solutions. One consequence is a new rigidity result for translators. Another is that, in dimension two, solutions are necessarily reflection symmetric across the mid-plane of their slab.

\end{abstract}

\maketitle
\thispagestyle{empty}

\tableofcontents

\section{Introduction}

%
%


The mean curvature flow arose in the 1950s as a physical model for the motion of grain boundaries in annealing metals \cite{Mullins,vonNeumann} and has been widely studied by mathematicians since the early 1980s, beginning with the foundational works of Brakke \cite{Brakke} and Huisken \cite{Hu84}, the latter initiating a programme analogous to Hamilton's programme for the Ricci flow \cite{HamILTON82}. 

Our focus here is on \textsc{ancient solutions}: those solutions which have existed for an infinite amount of time in the past.  Ancient solutions arise naturally in the study of singularities of the flow \cite{HamiltonPinched,HuSi99a,HuSi99b,Wh03}, and a deep understanding of them will have profound implications for the continuation of the flow through singularities (cf. \cite{HK2,HuSi09,Naff,NguyenSurgery}). The study of \emph{convex} ancient solutions is particularly pertinent in this context, since, by work of Huisken--Sinestrari \cite{HuSi99b,HuSi99a} (see also \cite{HK1,Wh03}), ancient solutions arising from singularities in \emph{mean convex} mean curvature flow are necessarily convex.

Ancient solutions also model the ultra-violet regime in certain quantum field theories, and early research on ancient solutions was undertaken by physicists in this context \cite{Bakas2,Lukyanov1,Lukyanov2}. 

Ancient solutions are known to exhibit rigidity phenomena resembling those of their elliptic counterparts: minimal hypersurfaces. For example, the shrinking sphere is the only ancient solution satisfying certain scale invariant geometric bounds \cite{BLT3,DHS,HaHe,HuSi15,L17,KevinLiouville,Wa11} (cf. \cite{BryanIvaki,BrIvSc,BrLo,ChMa,LambertLotaySchulze,LaLy,LynchNguyen,RisaSinestrari1,RisaSinestrari2}). Moreover, by a fundamental result of Wang \cite{Wa11}, no convex ancient solution having bounded curvature on each timeslice 
sweeps out a halfspace or a wedge --- they either sweep out all of space or are confined to the region between two stationary parallel hyperplanes (cf. \cite{ChiniMollerAncient,ChiniMollerTranslators}. See also \cite{BLT4}). The shrinking sphere is an example which sweeps out all of space. A well-known example which sweeps out a slab region is the Angenent oval solution in the plane \cite{Ang92} (referred to as the ``paperclip" solution in the physics literatire).

So it is natural to ask \emph{under which conditions can (convex) ancient solutions be meaningfully classified?} The shrinking circles and the Angenent ovals are the only examples in the plane with bounded, convex timeslices \cite{DHS}. If we relax the boundedness hypothesis, then the only additional examples are the stationary lines and the Grim Reapers (referred to as ``hairpin" solutions in the physics literature) \cite{BLT3}. On the other hand, if the convexity hypothesis is relaxed, there exist a plethora of non-trivial\footnote{Of course, all minimal hypersurfaces are examples, as are all solitons whose motion is a combination of translation, negative dilation, and rotation.} examples: 
``ancient trombone'' solutions in the plane can be constructed by ``gluing'' together an essentially arbitrary family of alternately oriented Grim Reapers along their common asymptotes \cite{AngenentYou,JohnMan} and every unstable mode of a compact, entropy unstable shrinker gives rise to a non-trivial ancient solution \cite{ChMa}. In \cite{BLM}, an explicit construction of a non-trivial ``ancient doughnut'' is given and in \cite{MramorPayne} a large class of explicit non-trivial examples evolving out of unstable minimal hypersurfaces are constructed.

In higher dimensions, a family of examples with bounded, convex timeslices which sweep out all of space (\textsc{ancient ovaloids}) have been constructed \cite{Ang12,HaHe,Wa11,Wh00}. In a series of remarkable recent papers, Angenent--Daskalopoulos--\v Se\v sum \cite{ADS1,ADS2} and Brendle--Choi \cite{BrendleChoi1,BrendleChoi2} proved that the only convex ancient solutions which are \emph{uniformly two-convex} and \emph{noncollapsing} are the shrinking sphere, the admissible shrinking cylinder, the admissible ancient ovaloid and the bowl soliton. There are also higher dimensional examples with bounded, convex timeslices which only sweep out slab regions (\textsc{ancient pancakes}) \cite{BLT1,Wa11}. We recently proved that there is only one rotationally symmetric ancient pancake, and obtained a precise asymptotic description of it \cite{BLT1}. This solution will play an important role here.  

In the noncompact setting, much less is known: aside from products of lower dimensional examples with lines, the only convex examples known are the translating solitons. Indeed, White has conjectured that the only convex \emph{eternal} solutions are the translators \cite[Conjecture 1]{Wh03}. 
Recent work on the construction and classification of convex translating solutions will also play an important role here \cite{BLT2,SX} (see also \cite{BL17,Ha15,HIMW,Wa11}).

According to Huisken and Sinestrari \cite{HuSi15}, ``[h]euristic arguments suggest that in higher dimensions many more compact convex ancient solutions may be constructed by appropriately gluing together lower dimensional translating solutions for $t\to-\infty$.'' We shall verify this picture by constructing convex ancient (translating) solutions to mean curvature flow out of any (unbounded) regular polytope, and partly classifying them in terms of their backwards limits. 
We shall also obtain ancient (translating) solutions out of any bounded (unbounded) simplex (which dispel the common belief that convex ancient solutions should be highly symmetric) as well as eternal solutions which do not evolve by translation (resolving \cite[Conjecture 1]{Wh03} mentioned above in the negative). 

All of the solutions we construct are reflection symmetric across the midplane of their slabs. As we shall see, this symmetry turns out to be necessary, at least in dimension two.

A very similar picture has been emerging concurrently in the context of the Ricci flow (see \cite{AngenentBrendleDaskalopoulosSesum,BamlerKleiner,BrendleRicci2014,BrendleHuiskenSinestrari,BryantSO3,DHSRicci,FateevOnofriZamolodchikov,HamiltonHarnackRicci,Ni2005,Ni2010,Perelman1}) although many of the corresponding questions remain open.

Before beginning our investigation, it is worth noting that many of the usual tools, such as curvature pinching, gradient estimates, noncollapsing\footnote{In the sense of Sheng--Wang \cite{ShWa09} and Andrews \cite{An12,ALM13}, say.}, parabolic rescaling (``blow-up'') arguments, and the monotonicity formula (for Gaussian area) are either not useful, or not even available in the setting studied here. Nor do we make use of any rotational symmetry hypothesis, as in our previous papers \cite{BLT1,BLT2}; as such, we are also unable to exploit the evolution equation for enclosed area, which was of great utility in \cite{BLT1,BLT3}. Our main tools are correspondingly limited: the scalar maximum principle, the avoidance principle, the differential Harnack inequality, and curvature comparison arguments, which we exploit repeatedly. In particular, the existence (and detailed description) of the rotationally symmetric example constructed in \cite{BLT1} seems to be crucial.

\subsection{Outline of the paper} In \S \ref{sec:prelims}, we recall some preliminary results on convex sets, mean curvature flow, and Alexandrov reflection that will be needed in the sequel. In particular, we exploit the differential Harnack inequality to obtain some not so well-known, but fairly elementary, results about convex ancient solutions and translators in $\R^{n+1}$. These results lead us naturally to consider their \textsc{squash-downs}, which play a central role in our analysis. 

In \S \ref{sec:existence}, we construct ancient (translating) examples with prescribed squash-downs --- first, we obtain a solution out of any bounded (unbounded) regular polytope 
(Theorems \ref{thm:bounded polygonal pancakes} and \ref{thm:unbounded polygonal pancakes}). We then show that the symmetry of the squash-down is not necessary, at least for polytopes with the minimal number of faces (Theorems \ref{thm:irregular polygonal pancakes bounded} and \ref{thm:irregular polygonal pancakes unbounded}). We then study the asymptotic translators of these solutions (Proposition \ref{prop:asymptotics}). We find that they are of the correct width and that their squash-downs are related to the squash-down of their ``parent'' solution in the obvious way. When $n=2$, this gives a canonical decomposition into asymptotic translators. In the bounded case, this affirms the heuristic described by Huisken and Sinestrari \cite{HuSi15}. In the unbounded case, we find that the same heuristic holds for translators. We note that the only non-entire ancient solution with bounded convex timeslices known previously was the rotationally symmetric example \cite{BLT1,Wa11}, and the only non-entire convex translating solutions known previously were rotationally symmetric with respect to the $(n-1)$-dimensional subspace parallel to the slab and orthogonal to the translation direction \cite{BLT2,SX}. 


In \S \ref{sec:squash-down}, we prove that 
the squash-down of a solution which sweeps out a slab $(-\frac{\pi}{2},\frac{\pi}{2})\times \R^{n}$ of width $\pi$ necessarily \emph{circumscribes} $\{0\}\times S^{n-1}$ (Theorem \ref{thm:circumscribed}). 

In \S \ref{sec:forward asymptotics}, we show that the \emph{forward} squash-down of every convex eternal solution which sweeps out a slab is the ``exscribed body'' determined by its (backward) squash-down (Theorem \ref{thm:squash-up}). As a consequence, we obtain a new rigidity result for translators: every convex ancient solution in a slab (with bounded curvature on each timeslice) whose squash-down is a cone necessarily moves by translation (Corollary \ref{cor:wedge implies translation}). 

We also obtain, in Sections \ref{sec:squash-down} and \ref{sec:forward asymptotics} structure results for the backward and forward asymptotic translators (Propositions \ref{prop:converging sequences} and \ref{prop:converging sequences nc}) which, in particular, are used in Section \ref{sec:reflection symmetry} to prove reflection symmetry in the two dimensional case.

In \S \ref{sec:eternal}, we construct eternal solutions with squash-down equal to the circumscribed truncation of any regular, circumscribed cone. In particular, these examples do not evolve by translation, and hence provide counterexamples to \cite[Conjecture 1]{Wh03}. Roughly speaking, these solutions arise from a family of ``generalized flying wing'' translators with common asymptotes emerging at $t=-\infty$ and coalescing into a single translator at $t=+\infty$, preserving their total exterior dihedral angles.

We conclude, in \S \ref{sec:reflection symmetry}, with a proof that every ancient solution in $\R^3$ which is confined to a slab region is invariant under reflection across its mid-hyperplane (Theorem \ref{thm:reflection symmetry}). The proof relies on the description of the asymptotic translators obtained in \S\S\ref{sec:squash-down}-\ref{sec:forward asymptotics}.

\subsection*{Acknowledgements}
We would like to thank Ben Andrews for bringing the problem to our attention, Brett Kotschwar for suggesting the ``doubling'' method, which proved so effective in the construction and analysis of the noncompact examples in \S \ref{sec:existence} and \S \ref{sec:eternal} and Romanos Diogenes Malikiosis for useful conversations on convex geometry.

T.~Bourni and M.~Langford thank the mathematical research institute MATRIX in Australia where part of this research was undertaken and G.~Tinaglia thanks the Department of Mathematics at the University of Tennessee, Knoxville, for their hospitality during the academic year 2018/19.

M.~Langford acknowledges the support of an Australian Research Council DECRA fellowship.

The first paragraph of the abstract, quoted directly from \cite{Democritus}, paraphrases several of the precepts of Democritean atomism. Their connection with ancient solutions to mean curvature flow is not to be taken literally.
\section{Preliminaries}\label{sec:prelims}

\subsection{Convex hypersurfaces} A hypersurface $\M\subset \R^{n+1}$ is \textsc{convex} if it bounds a convex body. A \emph{smooth}, convex hypersurface is \textsc{locally uniformly convex} if its shape operator is everywhere positive definite. Since the shape operator of $\M$ is, up to identification of the parallel hyperplanes $T_p\M\cong T_{G(p)}S^n$, the differential of the Gauss map $G:\M\to S^n$, a convex, locally uniformly convex hypersurface can be parametrized by the inverse $X:G(\M)\to \M$ of $G$. This parametrization is known as the \textsc{Gauss map parametrization}. It is closely related to the \textsc{support function} $\sigma:S^n\to \R$ of $\M$, which we recall is defined by
\[
\sigma(z)\doteqdot \sup_{p\in \M}\left\langle p,z\right\rangle.
\]
Equivalently, $\sigma(z)$ is the distance to the origin of the boundary of the supporting halfspace with outer normal $z$ (taken to be $\infty$ if no such halfspace exists). Indeed, for a convex, locally uniformly convex hypersurface,
\[
X(z)=\sigma(z)z+\overline\nabla \sigma(z)\,,
\]
where $\overline \nabla$ is the gradient operator induced by the standard metric on the sphere $S^n$.

\subsection{Mean curvature flow of convex hypersurfaces} A family $\{\M_t\}_{t\in (\alpha,\,\omega)}$ of smooth hypersurfaces $\M_t$ of $\R^{n+1}$ \textsc{evolves by mean curvature flow} if it admits about each point a family $X:U\times I\to \R^{n+1}$ of local parametrizations $X(\cdot,t):U\to \M_t$ satisfying
\[
\partial_tX=-H\nu\,,
\]
where $\nu$ is a local choice of unit normal field and $H=\operatorname{div} \nu$ is the corresponding mean curvature. 

A family $\{\M_t\}_{t\in (\alpha,\,\omega)}$ of \emph{convex, locally uniformly convex} hypersurfaces $\M_t$ in $\R^{n+1}$ evolves by mean curvature flow if and only if the corresponding family of support functions $\sigma(\cdot,t):S^n\to\R$ satisfies
\begin{equation}\label{eq:MCF support}
\partial_t\sigma(z,t)=-H(z,t)
\end{equation}
for each $z\in G(\M_t)$, where $H(z,t)$ is the mean curvature at the point of $\M_t$ whose normal direction is $z$.

We will make use of both the ``standard'' and Gauss map parametrizations. No confusion should arise in conflating the two, since we denote points of $S^n$ by $z,w,\dots$, and points in the domain of a fixed parametrization by $p,q,\dots$. An important feature of the Gauss map parametrization is the observation that the differential Harnack inequality \cite{HamiltonHarnack} (which applies to all locally uniformly convex solutions $\{\M_t\}_{t\in(\alpha,\omega)}$ with bounded curvature on each timeslice $\M_t$) takes the simple form \cite{AndrewsHarnack}
\begin{equation}\label{eq:Harnack nonancient}
\partial_t\left(\sqrt{t-\alpha}\,H(z,t)\right)\ge 0\,.
\end{equation}
In particular, if $\alpha=-\infty$, then
\begin{equation}\label{eq:Harnack}
\partial_tH(z,t)\ge 0\,.
\end{equation}
Moreover, the inequality \eqref{eq:Harnack} is strict unless the solution moves purely by translation \cite{HamiltonHarnack}; that is,
\[
\M_{t+s}=\M_t+s\vec v
\]
for some $\vec v\in \R^{n+1}$, called the \textsc{bulk velocity}\footnote{We note that the bulk velocity of a convex translator may not be unique. Indeed, it is unique if and only if the translator is locally uniformly convex.} of the solution. The timeslices of a translating solution with bulk velocity $\vec v$ satisfy
\[
H(z)=-\inner{\vec v}{z}\,.
\]
Solutions to this equation are called \textsc{translators}. A simple but important observation is the fact that the Grim Reapers and stationary lines are the only  translating solutions to the mean curvature flow in the plane.

The differential Harnack inequality \eqref{eq:Harnack} is an indispensable tool in the study of ancient solutions to mean curvature flow (with bounded curvature on each timeslice). It implies, in particular, that the  limit
\begin{equation}\label{eq:H_limit}
H_\ast(z)\doteqdot \lim_{t\to-\infty}H(z,t)
\end{equation}
exists for each $z\in G_\ast$, where
\[
G_\ast\doteqdot \bigcup_{s<\omega}\bigcap_{t\le s}G(\M_t)\,.
\] 

The flow equation \eqref{eq:MCF support} and the differential Harnack inequality \eqref{eq:Harnack} imply that the support function $\sigma$ of a convex, locally uniformly convex ancient solution is locally concave with respect to $t$. In particular, this implies that
\begin{equation}\label{eq:forward backward limits}
-H(z,s)\le \frac{\sigma(z,t)-\sigma(z,s)}{t-s}\le -H(z,t)
\end{equation}
for any $t<s<\omega$ and any $z\in \cap_{\tau\in [t,s]} G(\M_\tau)$. Taking $t\to-\infty$ and then $s\to-\infty$ yields
\begin{equation}\label{H_ast=sigma_ast}
\lim_{t\to-\infty}H(z,t)\doteqdot H_\ast(z)=\sigma_\ast(z)\doteqdot \lim_{t\to-\infty}\frac{\sigma(z,t)}{-t}
\end{equation}
for $z\in G_\ast$. The function $\sigma_\ast$ coincides on $G_\ast$ with the support function of the limiting convex region
\[
\Omega_\ast\doteqdot \lim_{t\to-\infty}\frac{\Omega_t}{-t}\,,
\]
where $\Omega_t$ is the convex body bounded by $\M_t$. We extend $\sigma_\ast$ to $S^n$ so that the two functions agree everywhere. 

We refer to the degenerate convex set $\Omega_\ast$ as the \textsc{squash-down} of $\{\M_t\}_{t\in(-\infty,\,\omega)}$. If $\{\M_t\}_{t\in(-\infty,\,\omega)}$ lies in a slab, $(-\frac{\pi}{2},\frac{\pi}{2})\times \R^n$ say, then 
$\Omega_\ast$ lies in the hyperplane $\{0\}\times\R^n$.

By \eqref{H_ast=sigma_ast}, $\sigma_\ast(z)\ge 1$ for $z\in G_\ast\cap(\{0\}\times\R^n)$:

\begin{proposition}\label{prop:sigma_ast ge 1}
Let $\{\M_t\}_{t\in(-\infty,\,\omega)}$ be a convex ancient solution to mean curvature flow in the slab $(-\frac{\pi}{2},\frac{\pi}{2})\times\R^{n}$. If $\sup_{\M_t}\vert \mathrm{II}\vert<\infty$ for each $t\in(-\infty,\omega)$, then
\[
\sigma_\ast(z)\ge 1
\]
for all $z\in G_\ast\cap(\{0\}\times\R^n)$, where $\sigma_\ast$ is the support function of the squash-down.
\end{proposition}
\begin{proof}
Suppose that $h_\ast \doteqdot \sigma_\ast(z) = H_\ast(z)$ is \emph{less than} 1 for some $z\in \{0\}\times S^{n-1}$ and let $\Sigma$ be any asymptotic translator corresponding to the normal direction $z$. Note that $\Sigma$ lies in a parallel slab of width at most $\pi$ and its bulk velocity $\vec v$ satisfyies $-\displaystyle\inner{\vec v}{z}=H_\ast(z)<1$. 
On the other hand, the oblique Grim hyperplane with bulk velocity $\vec v$ which contains the normal $z$ lies in a slab of width \emph{greater than} $\pi$. Since $\Sigma$ lies in the half-slab $\{p\in(-\frac{\pi}{2},\frac{\pi}{2})\times\R^n:\inner{p}{z}\le 0\}$, the oblique Grim hyperlane may be translated parallel to the slab until it is tangent to $\Sigma$ from the outside, in contradiction with the strong maximum principle.
\end{proof}

Observe that the squash-down of the rotationally symmetric ancient pancake is $\Omega_\ast=\{0\}\times S^{n-1}$ \cite{BLT1}, while the squash-down of a Grim hyperplane is a halfspace which supports $\{0\}\times S^{n-1}$. 
For more general translators, we obtain the following:

\begin{proposition}\label{prop:translator squash-down}
The squash-down of a convex translator\footnote{Of course, we really mean the squash-down of the corresponding translating solution $\{\Sigma+t\vec v\}_{t\in (-\infty,\infty)}$ to mean curvature flow.} $\Sigma$ of bulk velocity $\vec v$ is the cone
\[
\Omega_\ast=\bigcap_{z\in G(\Sigma)}\{p\in\{0\}\times \R^n:\inner{p}{z}\le -\inner{z}{\vec v}\}\,.
\]
\end{proposition}
\begin{proof}
By \eqref{H_ast=sigma_ast} and the translator equation,
\[
-\inner{z}{\vec v}=H_\ast(z)=\sigma_\ast(z)\doteqdot \sup_{p\in \Omega_\ast}\inner{z}{p}
\]
for every $z\in G_\ast=G(\Sigma)$. 
\end{proof}

For convex \textsc{immortal} solutions $\{\M_t\}_{t\in(\alpha,\infty)}$, one can similarly study the \emph{forward} asymptotics. The limit
\[
\Omega^\ast\doteqdot \lim_{t\to\infty}\frac{1}{t}\Omega_t\,,
\]
if it exists, is called the \textsc{forward squash-down} of $\{\M_t\}_{t\in(\alpha,\infty)}$. For convex \textsc{eternal} solutions $\{\M_t\}_{t\in(-\infty,\infty)}$, we will sometimes refer to the squash-down as the \textsc{backward squash-down}. 

By the differential Harnack inequality, \eqref{eq:Harnack}, the limit
\[
H^\ast(z)\doteqdot \lim_{t\to\infty}H(z,t)
\]
exists for all $z\in G^\ast$ (although it might be infinite) on a convex eternal solution with bounded curvature on each timeslice, where
\[
G^\ast\doteqdot \bigcup_{s>-\infty}\bigcap_{t\ge s} G(\M_t)\,.
\]
By \eqref{eq:forward backward limits},
\begin{equation}\label{eq:forward sigma limit}
-H^\ast(z)=\sigma^\ast(z)\doteqdot \lim_{t\to\infty}\frac{\sigma(z,t)}{t}
\end{equation}
and $\sigma^\ast$ coincides on $G^\ast$ with the support function of $\Omega^\ast$. As for the backward squash-down, we extend $\sigma^\ast$ so that the two agree everywhere.

Finally, we use \eqref{eq:H_limit} and the rigidity case of the Harnack inequality to obtain that solutions converge to translators after space-time translating so that the normal at the origin at time zero is fixed (cf.  \cite[Lemma 5.2]{BLT1}).

\begin{lemma}\label{lem:backwards asymptotic translators}
Let $\{\M_{t}\}_{t\in(-\infty,\omega)}$ be a convex ancient solution to mean curvature flow satisfying $\sup_{\M_t}\vert{\operatorname{II}}\vert<\infty$ for each $t$ and whose Gauss image is constant in time. For each $z\in G_\ast$ and each sequence of normals $z_j\in G_\ast$ with $z_j\to z$,
there exist a sequence of times $s_j\to-\infty$, a linear subspace $E\subset \R^{n+1}\cong E^\perp\times E$, and a locally uniformly convex translating solution $\{\Sigma_t\}_{t\in(-\infty,\infty)}$ in $E^\perp$ such that $z\in G(\Sigma_t)$, the flows $\{\M_{s,t}\}_{t\in(-\infty,\,\omega-s)}$ defined by
\[
\M^j_{t}\doteqdot \M_{t+s_j}-X(z_j,s_j)
\]
converge locally uniformly in the smooth topology  to $\{\Sigma_t\times E\}_{t\in(-\infty,\infty)}$, the limit $\lim_{i\to\infty}\frac{X(z_j, s_j)}{-s_j}=:X_\ast(z)$ exists, and the bulk velocity of the limit translator, $\vec v\in E^\perp$, satisfies
\begin{equation}\label{MCoflimittr}
-\inner{\vec v}{w}=\inner{X_\ast(z)}{w}=H_\ast(w)\;\; \text{for all}\;\; w\in G(\Sigma_t)\,.
\end{equation}
\end{lemma}
\begin{proof}
By the differential Harnack inequality, the mean curvature of the sequence is bounded on any time interval of the form $(-\infty,T]$, $T<\infty$, for $j$ sufficiently large. Since the solution is convex, this implies a bound for the second fundamental form. Since the solutions pass through the origin at time zero, standard bootstrapping and compactness results yield a convex, eternal limit solution along some sequence of times $s_j\to-\infty$. By the splitting theorem for the second fundamental form, we can find a linear subspace $E\subset \R^{n+1}\cong E^\perp\times E$ and a locally uniformly convex eternal solution $\{\Sigma_t\}_{t\in(-\infty,\infty)}$ to mean curvature flow in $E^\perp$ such that the limit is of the form $\{\Sigma_t\times E\}_{t\in(-\infty,\infty)}$. 

We first show that $\{\Sigma_t\}_{t\in(-\infty,\infty)}$ is a translating solution in the special case\footnote{In this case, we do not require that the Gauss image is constant in time. However, we will see later (Proposition \ref{constant Gauss}) that this is always the case.} $z_j\equiv z$.
Indeed, since the solutions converge smoothly as convex bodies to the eternal limit, the support functions converge. Thus,
\[
\sigma_\infty(z,t)=\lim_{j\to\infty}(\sigma(z,t+s_j)-\sigma(z,s_j))\,,
\]
where $\s_\infty(\cdot, t)$ denotes the support function of $\Sigma_t\times E$. 
Since $\sigma$ is a concave function of $t$, the limit on the right is a linear function of $t$. In fact, by \eqref{eq:forward backward limits} it is given by $-H_\ast(z) t$. Restricting to the locally uniformly convex cross-section, we conclude that
\[
\partial_tH_\infty(z,0)=0\,,
\]
where $H_\infty(\cdot, t)$ denotes the mean curvature of $\Sigma_t$. The rigidity case of the Harnack inequality then implies that the cross-section evolves by translation (in the special  case $z_j\equiv z$).

Next, we show that \eqref{MCoflimittr} holds in the general case. 
Note that by construction, $\{\Sigma_t\}_{t\in(-\infty,\infty)}$ satisfies $z\in G(\Sigma_0)$ and 
\[
H(z,0)=\lim_{j\to\infty}H(z_j, s_j)\,.
\]
Let $w$ be any vector in $G(\Sigma_0)$. Then, for $|t|$ sufficiently small, $w\in G(\Sigma_t)$ and hence there is a sequence of normals $w_j^t\in G(\M_{t+s_j})=G_\ast$ such that $X(w_j^t, t+s_j) -X(z_j, s_j)$ converges to a point on $\Sigma_t$ whose normal is $w$. Moreover,
\[
H(w,t)=\lim_{j\to\infty} H(w^t_j, t+s_j)\,.
\]

We next claim that 
\begin{equation}\label{HHast}
\lim_{j\to\infty}H(w^t_j, t+s_j)= H_\ast (w)\,.
\end{equation}

First note that, by the differential Harnack inequality and the continuity of $H_\ast(\cdot)$,
\begin{equation}\label{eq:in fact 1}
\lim_{j\to\infty}H(w^t_j, t+s_j)\ge \lim_{j\to\infty}H_\ast(w^t_j)= H_\ast(w)\,.
\end{equation}
On the other hand, by \eqref{eq:forward backward limits}, we have that
\[
\lim_{j\to\infty}H(w^t_j, t+s_j)\le \lim_{j\to\infty}\frac{\s(w^t_j, t+s_j)}{-t-s_j}.
\]
We claim that 
\[
 \lim_{j\to\infty}\frac{\s(w^t_j, t+s_j)}{-t-s_j}\le \s_\ast(w)\,.
\]
Note first that for any sequence $t_j\to-\infty$
\[
\limsup_{j\to\infty}\,(-t_j)^{-1}|X(z_j,t_j)|<\infty\,. 
\]
Indeed, if not, then, after passing to a subsequence,
\begin{equation}\label{eq:limit exists}
(-t_j)^{-1}|X(z_j,t_j)|\underset{j\to\infty}{\to}\infty\,.
\end{equation}
Since $(-t)^{-1}\Omega_t\to \Omega_\ast$, for any point $X_\ast(z)$ in $\Omega_\ast$ which lies in a support hyperplane whose normal is $z$, there exists a sequence of normals $\widehat z_j\in G_\ast$ with $\widehat z_j\to z$ such that $(-t_j)^{-1}X(\widehat z_j, t_j)\to X_\ast(z)$. By convexity, this implies that $(-t_j)^{-1}(X(z_j, t_j)- X(\widehat z_j, t_j))$ converges to a half line orthogonal to $z$ which is contained in $\Omega_\ast$. However, this is impossible since $z\in G_\ast$. We conclude that, after passing to a subsequence,
\begin{equation}\label{ptwlimit}
(-s_j)^{-1}X(z_j,s_j)\underset{j\to\infty}{\to} X_\ast(z)
\end{equation}
for some point $X_\ast(z)$ in $\Omega_\ast$ which lies in support hyperplane whose normal is $z$. 

Next, we compute
\[
\begin{split}
\s(w^t_j, t+s_j)&= \langle X(w^t_j, t+s_j), w^t_j\rangle\\
&= \langle X(w^t_j, t+s_j)-X(z_j, s_j), w^t_j\rangle+\langle X(z_j, s_j), w^t_j\rangle\,.
\end{split}
\]
Since $|X(w^t_j, t+s_j)-X(z_j, s_j)|$ is bounded uniformly in $j$, \eqref{ptwlimit} and $w^t_j\to w$ imply that
\begin{equation}\label{eq:in fact 2}
\lim_{j\to\infty}\frac{\s(w^t_j, t+s_j)}{-t-s_j}= \lim_{j\to\infty}\frac{\langle X(z_j, s_j), w^t_j\rangle}{-t-s_j}=\langle X_\ast(z),w\rangle\le\s_\ast(w)\,.
\end{equation}
Thus we obtain \eqref{HHast} and in fact (by \eqref{eq:in fact 1} and \eqref{eq:in fact 2}) also \eqref{MCoflimittr}.
\end{proof}

We refer to any translator obtained in this way as an \textsc{asymptotic translator} of $\{\M_{t}\}_{t\in(-\infty,\,\omega)}$ (corresponding to $z\in G_\ast$). 

If $\{\M_{t}\}_{t\in(-\infty,\,\infty)}$ is itself a translator, we obtain asymptotic translators at spatial infinity.


\begin{lemma}\label{lem:backwards asymptotic translators2}
Let $\M$ be a convex translator. For each $z\in\overline{G(\M)}$ and each sequence of normals $z_j\in G(\M)$ converging to $z$, there exists a nontrivial linear subspace $E\subset \R^{n+1}\cong E^\perp\times E$, and a locally uniformly convex translator $\Sigma$ in $E^\perp$ such that the translators
\[
\M^j\doteqdot \M-X(z_j)
\]
converge locally uniformly in the smooth topology, along a subsequence, to $\Sigma\times E$. The bulk velocity of the limit, $\vec v\in E^\perp$, satisfies
\[
-\inner{\vec v}{z}=\lim_{j\to\infty}H(z_j)\,.
\]
\end{lemma}
\begin{proof}
Let $\{z_j\}_{j\in \N}$ be a sequence of normals $z_j\in G(\M)$ which converge to $z$. Without loss of generality, $\vert X(z_j)\vert\to\infty$ (since else, after passing to a subsequence, we obtain a finite translate of $\M$ in the limit). Passing to a subsequence, we may arrange that $\lim_{j\to\infty}\frac{X(z_j)}{\vert X(z_j)\vert}$ converges to some limit direction $w$. Since the translator is convex, the translator equation implies a bound for the second fundamental form. Since each $\M^j$ passes through the origin, standard bootstrapping and compactness results, and the splitting theorem, yield a subspace $E\subset \R^{n+1}\cong E^\perp\times E$ and a locally uniformly convex translator $\Sigma\subset E^\perp$ such that $\M^j\to \Sigma\times E$ locally uniformly in the smooth topology after passing to a further subsequence. The subspace $E$ is nontrivial since $\Sigma$ splits off the line $\R w$.  The bulk velocity of the limit, $\vec v\in E^\perp$, satisfies
\[
-\inner{\vec v}{z}=H(z)=\lim_{j\to\infty}H_j(z_j)=\lim_{j\to\infty}H(z_j)\,,
\]
where $H_j$ denotes the mean curvature of $\M^j$.
\end{proof}

We refer to any translator obtained in this way as an \textsc{asymptotic translator} of $\M$ (corresponding to $z\in \overline{G(\M)}$).

\subsection{Polytopes} The intersection of finitely many closed halfspaces in a finite dimensional linear space is called a (\textsc{convex}) \textsc{polytope}\footnote{In order to avoid cumbersome repetition of the word ``convex", we follow the convention of referring to \emph{convex polytopes} simply as \emph{polytopes}. Note that we do not require polytopes to be bounded --- Many sources reserve the term ``polytope'' for the bounded case, using the term \textsc{(convex) polyhedron} for the general (unbounded) case.} \cite{Polytopes}. The \textsc{dimension} of a polytope $P$ is the dimension of its \textsc{affine hull} (the smallest affine subspace which contains $P$). An $n$-dimensional polytope is called an $n$-\textsc{polytope}. A $2$-polytope is called a \textsc{convex polygon}. Note however that, in contrast to the common notion of polygon, polytopes are allowed to be unbounded according to the definition used here. 

A subset $f$ of a polytope $P$ is called a \textsc{face} of $P$ if $f=P\cap H$ for some supporting hyperplane $H$ of $P$. The faces of $P$ are themselves polytopes. A $k$-dimensional face is called a $k$-\textsc{face}. The $0$-faces of $P$ are called \textsc{vertices}, the $1$-faces \textsc{edges} and the $(n-1)$-faces \textsc{facets}. 
A \textsc{flag} of an $n$-polytope is a finite sequence $\{f_0,\dots, f_n\}$ of $j$-faces $f_j$ (exactly one for each $j=0,\dots,n$) such that $f_j\subset f_{j+1}$ for each $0\le j \le n-1$.

A \textsc{simplex} is a polytope with the least possible number of faces relative to its dimension after splitting off any lines. More precisely, a bounded polytope $P$ of dimension $k$ is a (\textsc{bounded}) \textsc{simplex} if it has $k+1$ vertices. An unbounded polytope is an (\textsc{unbounded}) \textsc{simplex} if it is a nontrivial cone whose link is a bounded simplex.

We say that a convex body $P$ \textsc{circumscribes} another convex body $S$ if $P$ is the intersection of a family of halfspaces which support $S$. If the linear space is normed and $S$ is the unit sphere, we simply say that $P$ is \textsc{circumscribed}. 

A polytope whose symmetry group acts transitively on its flags necessarily circumscribes a sphere. We will call a circumscribed polytope \textsc{regular} if its symmetry group acts transitively on its flags.

We denote by $\mathrm{P}^n_\ast$ the set of circumscribed polytopes in $\R^{n}$, and by $\overline{\mathrm{P}}{}^n_\ast$ the set of circumscribed convex bodies in $\R^n$. 

The \textsc{relative interior}, $\relint \Omega$, of a convex set $\Omega$ is its interior with respect to its affine hull. Analogously, $\relint \Omega$ of a convex set $\Omega\subset S^n$ is its interior with respect to the smallest subsphere (the intersection with $S^n$ of a linear subspace) in which it lies.

Define the \textsc{Gauss image} $G(\Omega)$ of a convex set $\Omega$ to be the set of unit outward normals to halfspaces which support $\Omega$. Note that $G(\Omega)$ is the whole sphere if $\Omega$ is bounded, whereas $G(\Omega)$ lies in a closed hemisphere if $\Omega$ is unbounded. We shall say that a circumscribed convex body $P\in \overline{\mathrm{P}}{}_\ast^n$ is \textsc{degenerate} if $G(P)$ lies in a closed hemisphere but in no open hemisphere. Else, we shall say that $P$ is \textsc{nondegenerate}. We will also say that a degenerate polytope $P$ is \textsc{fully degenerate} if $G(P)$ is a closed hemisphere (of any codimension) and \textsc{semi-degenerate} otherwise.



The \textsc{support cone} (or \textsc{tangent cone}) of a closed convex set $\Omega$ in $\R^n$ at a point $p\in\Omega$ is defined to be
\begin{equation}\label{def:tangent cone}
T_p\Omega\doteqdot \overline{\{\lambda(q-p):q\in\Omega,\lambda>0\}}\,.
\end{equation}

\subsection{The rotationally symmetric ancient pancake}\label{ssec:rotational pancake}

As it will feature frequently in the sequel, it is well to recall here some basic properties of the rotationally symmetric ancient pancake.

In \cite{BLT1} (cf. \cite{Wa11}), it was proved that, for each $n\in\N$, there exists a unique (modulo spacetime translation) rotationally symmetric ancient solution $\{\Pi_t\}_{t\in(-\infty,0)}$ to mean curvature flow which sweeps out $(-\frac{\pi}{2},\frac{\pi}{2})\times \R^n$. The one dimensional example is the well-known \textsc{Angenent oval} $\{\mathrm{A}_t\}_{t\in(-\infty,0)}$, which is given by
\[
\mathrm{A}_t\doteqdot \{(x,y)\in \R^2:\cos x=\mathrm{e}^t\cosh y\}.
\]
Alternatively, $\mathrm{A}_t$ can be described by the immersion $\theta\mapsto (x(\theta,t),y(\theta,t))$, where, setting $a(t)\doteqdot (e^{-2t}-1)^{-\frac{1}{2}}$,
\begin{subequations}
\begin{align}
x(\theta,t)\doteqdot{}& \arctan\left(\frac{\sin\theta}{\sqrt{\cos^2\theta+a^2(t)}}\right)\label{eq:Angenent oval x}\,,\\
y(\theta,t)\doteqdot{}& -t+\log\left(\frac{\sqrt{\cos^2\theta+a^2(t)}+\cos\theta}{\sqrt{1+a^2(t)}}\right).\label{eq:Angenent oval y}
\end{align}
\end{subequations}
 This parametrization is clockwise oriented, and the parameter $\theta\in \R/2\pi \Z$ corresponds to the (clockwise oriented) \textsc{turning angle} of $\mathrm{A}_t$. With respect to this convention, the unit tangent vector is given by $\tau(\theta,t)=(\cos\th, -\sin\th)$ and the outward pointing unit normal is given by $\nu(\theta,t)=(\sin\th, \cos\th)$. The curvature $\kappa(\cdot,t)$ of $\mathrm{A}_t$, with respect to $\theta$, is given by
\begin{equation}\label{eq:Angenent oval kappa}
\kappa(\theta,t)=\sqrt{\cos^2\theta+a^2(t)}\,.
\end{equation}

In higher dimensions, $\{\Pi_t\}_{t\in(-\infty,0)}$ is obtained by taking the limit as $R\to\infty$ of the old-but-not-ancient flows $\{\Pi^R_t\}_{t\in[\alpha_R,0)}$ obtained by evolving the rotation
\[
\Pi^R\doteqdot \big\{x(\theta,-R)e_1+y(\theta,-R)\phi:(\theta,\phi)\in [-\tfrac{\pi}{2},\tfrac{\pi}{2}]\times S^{n-1}\big\}
\]
of $\mathrm{A}_{-R}$ about the $x$-axis, and translating in time so that the singularity occurs (at the origin) at time zero.

The \textsc{horizontal displacement}
\[
\ell(t)\doteqdot \sigma(e_2,t)
\]
of $\{\Pi_t\}_{t\in(-\infty,0)}$ satisfies
\begin{equation}\label{eq:ancient pancake radius}
\ell(t)=-t+(n-1)\log(-t)+c_n+o(1)\;\;\text{as}\;\; t\to-\infty\,,
\end{equation}
where $c_n$ is a constant which depends only on $n$. By \eqref{eq:Angenent oval y}, $c_1=\log 2$.

For each $z\in\{0\}\times S^{n-1}$, the corresponding asymptotic translator is the Grim hyperplane (of width $\pi$): 
\begin{equation}\label{eq:pancake edge asymptotics}
\{\Pi_{t+s}-X(z,s)\}_{t\in(-\infty,-s)}\underset{C^\infty_{\mathrm{loc}}}{\longrightarrow}\{\rho_z(\Gamma_t\times \R^{n-1})\}_{t\in(-\infty,\infty)}\,,
\end{equation}
as  $s\to-\infty$, where $\{\Gamma_t\}_{t\in(-\infty,\infty)}$ is the Grim Reaper and $\rho_z$ is a rotation in $\{0\}\times\R^n$ which takes $-e_{2}$ to $z$.

\subsection{Alexandrov reflection}

Alexandrov reflection will be applied in \S \ref{sec:reflection symmetry} to  prove reflection symmetry of convex ancient solutions in slab regions. We make use of ``tilted'' hyperplanes, inspired by an argument of Korevaar, Kusner and Solomon \cite{KorevaarKusnerSolomon} in the context of constant mean curvature surfaces in $\R^3$.

Given a hyperplane $\pi\subset \R^{n+1}$ with unit normal $e$, denote by $\pi ^\perp\doteqdot\{he:h\in\R\}$ the line orthogonal to $\pi$. Given $h\in \R$ and $p\in \pi$, denote by $\pi_h$ the $\pi$-parallel hyperplane displaced by $h$, by $\Pi_h^+$ and $\Pi_h^-$ the corresponding closed upper and lower halfspaces, and by $\pi_p^\perp$ the orthogonal line passing through $p$; that is,
\[
\pi_h\doteqdot \pi+he\,,\,\,\Pi_h^+\doteqdot \bigcup_{s\ge h}\pi_s\,,\,\, \Pi_h^-\doteqdot \bigcup_{s\le h}\pi_s\,,\;\;\text{and}\;\;\pi_p^\perp\doteqdot p+\pi^\perp.
\]
For any set $G\subset \R^{n+1}$
we let $R_h(G)$ be its reflection through $\pi_h$:
\[
R_h(G)\doteqdot \{p+(h-r)e: p\in \pi\,,\,\,p+(h+r)e\in G\}\,.
\]

Consider now a convex ancient solution $\{\M_t\}_{t\in (-\infty, \o)}$  in $(-\frac{\pi}{2},\frac{\pi}{2})\times\R^n$ and denote by $\Omega_t$ the convex body bounded by $\M_t$. For each $t\in(-\infty, \o)$ and $p\in \pi\cap\Omega_t$, we let $h_+(p,t)$ and $h_-(p,t)$ be, respectively, the largest and smallest $h\in \R$ so that $p+he\in \M_t$, and define the function
\begin{equation}\label{afct}
\a(p,t)= \frac{h_+(p,t)+h_-(p,t)}{2}\,.
\end{equation}


By the strong maximum principle, $\alpha$ is maximized on the parabolic boundary of any bounded subset of space-time.
\begin{lemma}\label{lem:wedge}
Fix any plane $\pi$. Given any open $W\subset \R^{n+1}$ and any $[t,t_0]\subset \R$ such that $W\cap\pi\cap \Omega_s$ is bounded for all $s\in[t, t_0]$,
\[
\max_{W\times [t, t_0]}\a=\max\left\{\max_{\partial W\times [t, t_0]}\a, \max_{W\times\{t\}}\a\right\}\,.
\]
\end{lemma}
\begin{proof}
Assume that the lemma is not true. Then there exist $s_0\in (t, t_0]$ and $p_0\in (W\setminus \partial W)\cap \pi\cap \Omega_{s_0}$ such that
\begin{equation}\label{eq:reflection principle contra}
\a(p_0,s_0)= \max_{W\times [t, t_0]}\a>\max\left\{\max_{\partial W\times [t, t_0]}\a, \max_{W\times\{t\}}\a\right\}.
\end{equation}
Set  $\a_0\doteqdot \a(p_0,s_0)$. By the maximality of $\a_0$ we have, for all $s\in [t, t_0]$ and $p\in \pi\cap\Omega_s\cap W$,
\[
\frac{h_+(p,s)+h_-(p,s)}{2}\le \frac{h_+(p_0,s_0)+h_-(p_0,s_0)}{2}=\a_0
\]
and hence
\[
\a_0- (h_+(p, s)-\a_0)\ge h_-(p,s)
\]
with equality at $(p_0,s_0)$. That is, $R_{\a_0} (\M_s\cap \Pi^+_{\a_0})$ lies above $\M_s\cap \Pi^-_{\a_0}$ for all $s\in [t, t_0]$, while at the interior time $s_0$ they are tangent at the interior point $p_0$. The strong maximum principle now implies that they should coincide for all $s\in [t, s_0]$. That is, $\pi_{\a_0}$ is a plane of symmetry for $\M_s\cap W$ for each $s\in [t,s_0]$. But then $\a(p,s_0)=\a_0$ for $p\in \partial W$, contradicting \eqref{eq:reflection principle contra}.
\end{proof}

\section{Old-but-not-ancient solutions}\label{sec:old-but-not-ancient}


%
Given a circumscribed polytope $P\in \mathrm{P}_\ast^n$ in $\{0\}\times \R^n\subset \R^{n+1}$, denote by $F_P^k$ the set of $k$-faces of $P$. For each facet $f\in F^{n-1}_P$, denote by $z_f$ the point of contact of $f$ with $\{0\}\times S^{n-1}$, and consider the convex region $\Omega_f$ bounded by the Grim hyperplane $\Gamma_f\doteqdot \rho_f\Gamma$, where
\[
\Gamma\doteqdot \big\{(x,y,-\log\cos x):(x,y)\in (-\tfrac{\pi}{2},\tfrac{\pi}{2})\times\R^{n-1}\big\}
\]
is the `standard' Grim hyperplane and $\rho_f$ is a rotation in $\{0\}\times \R^n$ which maps $-e_{n+1}$ to $z_f$. Given $R>0$, consider the boundary $\M^R$ of the convex region
\[
\Omega^R\doteqdot \bigcap_{f\in F^{n-1}_P}(\Omega_f+Rz_f)\,.
\]
Observe that $\lim_{R\to \infty}\frac{1}{R}\Omega^R=P$.

\begin{figure}[h]
\centering
\includegraphics[width=0.42\textwidth]{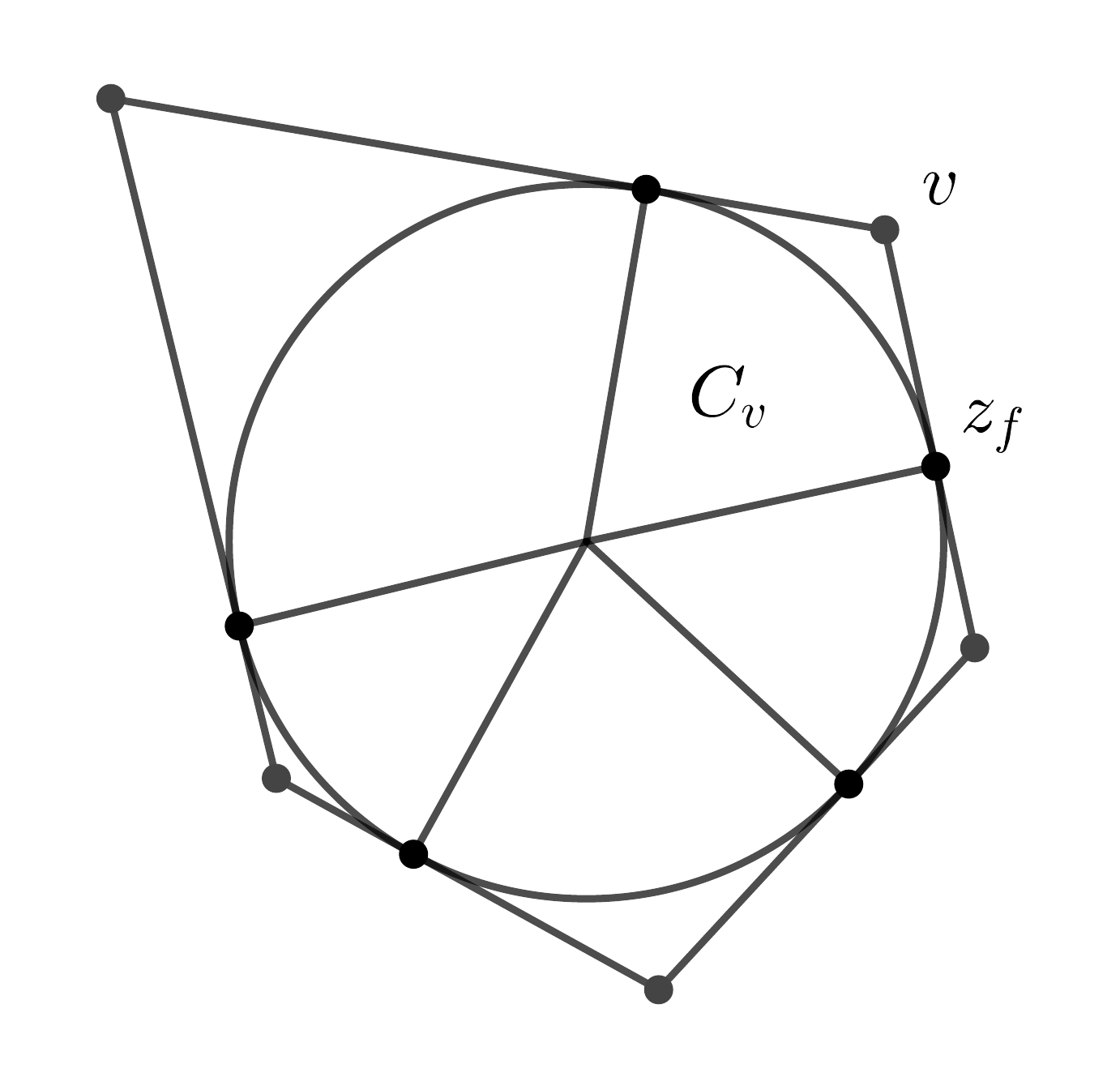} \includegraphics[width=0.54\textwidth]{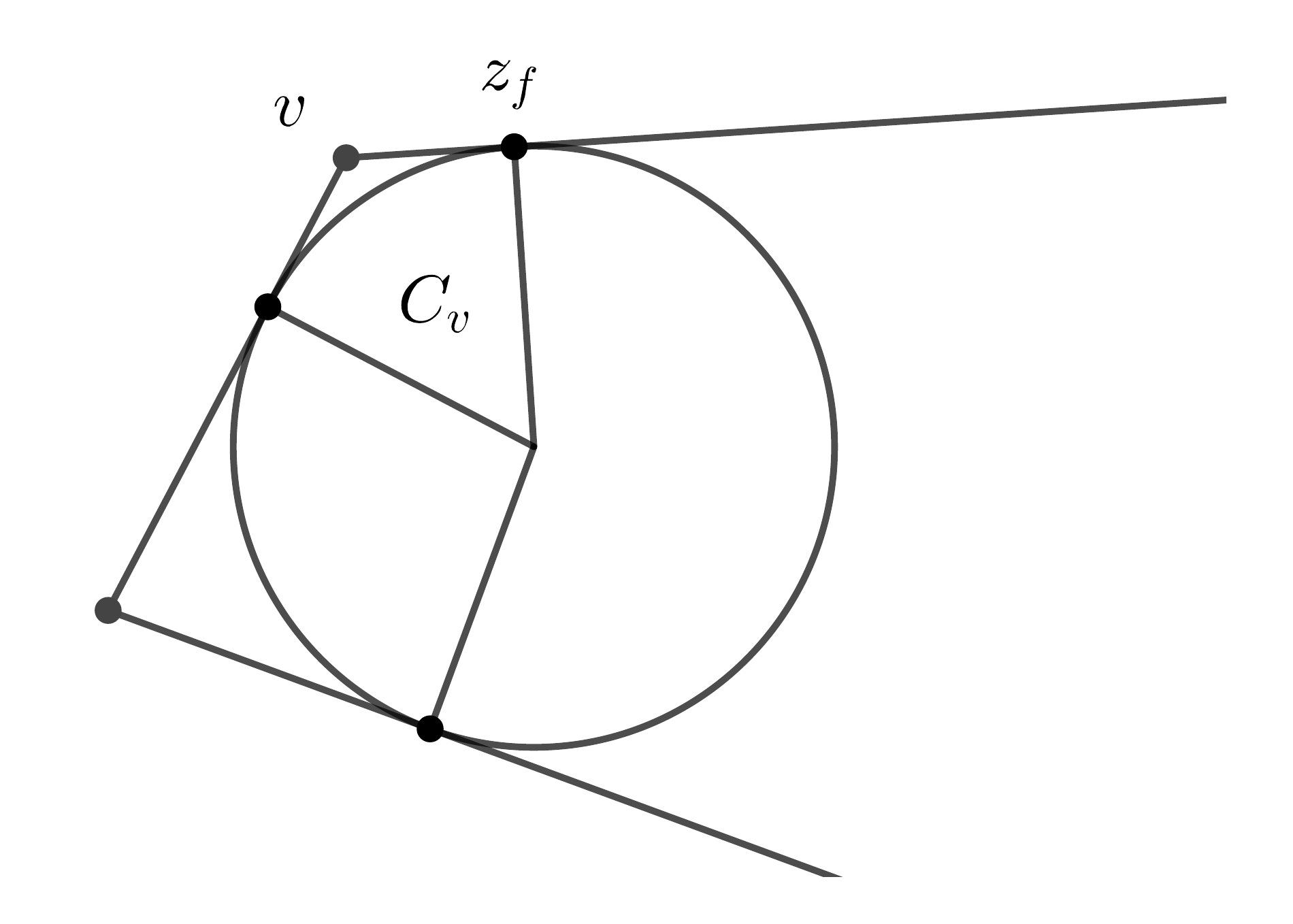}
\caption{\footnotesize Circumscribed polytopes.}\label{fig:circumscribed polygons}
\end{figure}

Denote by $\operatorname{reg}(\M^R)$ the regular set of $\M^R$.

\begin{lemma}\label{lem:H initial lower bound}
For each vertex $v\in F^0_P$,
\[
H_R(p)-\left\langle \nu_R(p),v\right\rangle\ge 0\;\;\text{for all}\;\; p\in \operatorname{reg}(\M^R)\,,
\]
where $H_R(p)$ is the mean curvature and $\nu_R(p)$ is the unit outward pointing normal to $\M^R$ at $p$.
\end{lemma}
\begin{proof}
Let $p\in\operatorname{reg}(\M^R)$ be a regular point of $\M_R$. Choose $f\in F^{n-1}_P$ so that $p\in\M^R\cap(\Gamma_f+Rz_f)$. 
Then, by construction, 
\[
H_R(p)=\inner{\nu_R(p)}{z_f}
\]
and $\nu_R(p)\in\operatorname{span}\{z_f,e_{1}\}$. Thus, for any $v\in F_P^{0}$,
\[
\inner{\nu_R(p)}{v}=\inner{\nu_R(p)}{z_f}\inner{v}{z_f}\le \inner{\nu_R(p)}{z_f}
\]
with equality if and only if $v\in F_P^{0}\cap (\Gamma_f+Rz_f)$. 
%
\end{proof}

We now obtain ``old-but-not-ancient'' solutions $\{\M^R_t\}_{t\in[0, T_R)}$ through the evolution of the hypersurfaces $\M^R$ by mean curvature flow.\footnote{In fact, to obtain true \emph{old}-but-not-ancient solutions, we first need to translate the time-interval of existence appropriately. This needs to be done very carefully if non-trivial limits are to be obtained. We will do it later, in two different ways.}




\begin{lemma}\label{lem:old-but-not-ancient}
Given $P\in \mathrm{P}^n_\ast$ and $R>0$, there exists a maximal solution $\{\M_{t}^R\}_{t\in [0,T_R)}$ to mean curvature flow which is smooth and locally uniformly convex at interior times, converges in $C_{\mathrm{loc}}^{0,1}$ to $\M^R$ as $t\to 0$. The final time $T_R$ satisfies $0\le T_R\le R$ if $P$ is bounded or degenerate, and $T_R=\infty$ if $P$ is unbounded but nondegenerate. In all cases,
\begin{equation}\label{eq:initial time R}
\liminf_{R\to\infty}\frac{T^0_R}{R}\ge 1\,,
\end{equation}
where $T^0_R\le T_R$ is the time that the solution reaches the origin.
\end{lemma}
\begin{proof}
When $P$ is bounded, existence of a maximal solution to mean curvature flow which is smooth at interior times and converges in $C^{0,1}$ to $\M^R$ at the initial time $t=0$ follows readily from the convexity of $\M^R$ and the interior estimates of Ecker and Huisken \cite{EckerHuisken91}. By Huisken's theorem \cite{Hu84}, the solution contracts to a point at its maximal time, $T_R$. Note that $T_R\le R$ since the initial configuration of Grim hyperplanes reach the origin at time $R$. 
By \eqref{eq:ancient pancake radius} and \eqref{eq:pancake edge asymptotics}, we can choose $t_R=-R+o(R)$ as $R\to\infty$ so that the time $t_R$ slice $\Pi_{t_R}$ of the rotationally symmetric ancient pancake is enclosed by $\M^R$. The desired lower bound for $T^0_R$ then follows from the avoidance principle.

When $P$ is unbounded, we may assume that it does not split off a line (else, the desired solution is obtained by taking products with lines of the solution corresponding to the cross-section of $P$). We may then choose $e\in \{0\}\times S^{n-1}$ so that the Gauss image $G(P)$ of $P$ (as a convex body in $\{0\}\times \R^n$) is contained in the closed lower hemisphere $\{z\in \{0\}\times S^{n-1}:\inner{z}{e}\leq 0\}$ and $-e\in \operatorname{int}(G(P))$. Consider, for each height $L>0$, the boundary $\M^{R,L}$ of the intersection of $\Omega^{R}$ with its reflection across the plane $\{p:\inner{p}{e}=L\}$. 
Since $\M^{R,L}$ is compact, the Ecker--Huisken estimates yield a maximal solution $\{\M^{R,L}_t\}_{t\in[0,T_{R,L})}$ for each $R$ and $L$ which is smooth at interior times and converges in $C^{0,1}$ to $\M^{R,L}$ at time $0$. By Huisken's theorem, $\{\M^{R,L}_t\}_{t\in[0,T_{R,L})}$ remains smooth until it contracts to a point at time $T_{R,L}$. If $G(P)$ lies in an open hemisphere, then $T_{R,L}\to\infty$ as $L\to \infty$ for all $R$ since, for $L$ sufficiently large, $\M^{R,L}$ encloses an arbitrarily early time-slice of the rotationally symmetric ancient pancake. On the other hand, if $G(P)$ does not lie in an open hemisphere, then $T_{R,L}\le R$ since $\M^R$ lies between a pair of parallel, oppositely oriented Grim hyperplanes which reach the origin, enclosing zero volume, after time $R$. The desired solution is then obtained by taking a limit as $L\to\infty$ (the lower bound for $T^0_R/R$ is obtained as in the bounded case). 


Local uniform convexity of $\{\M^R_t\}_{t\in(0,T_R)}$ at interior times follows from the splitting theorem for the second fundamental form.
\end{proof}

For the remainder of this section, we fix $P\in \mathrm{P}_\ast^n$ and $R>0$ and consider the corresponding old-but-not-ancient solution $\{\M^R_t\}_{t\in[0,T_R)}$.

When $P$ is bounded, the Gauss image of $\M^R_t$ is all of $S^n$ for all $t$, so the evolution equation \eqref{eq:MCF support} holds for all $z\in S^n$. In order to make full use of \eqref{eq:MCF support} in the unbounded case, we need to control the Gauss image. On the one hand, since 
the enclosed regions $\Omega_t^R$ satisfy $\Omega^R_s\subset \Omega^R_t$ for $t<s$, 
and $\M_t$ is locally uniformly convex unless it splits off a line, the Gauss image $G(\M_t^R)$ is nondecreasing: if $z\in G(\M_{t_0})$, $t_0>0$, then $z\in G(\M_t)$ for all $t>t_0$. We claim that $G(\M^R_t)$ is actually constant.

\begin{lemma}\label{lem:Gauss image}
The Gauss image $G(\M^R_t)$ is constant for $t\in(0,T_R)$.
\end{lemma}
\begin{proof}
The claim is clear in case $P$ is bounded, and follows from the monotonicity of $G(\M_t)$ in case $P$ is fully degenerate. 

In case $P$ is unbounded but nondegenerate, we may 
use the rotationally symmetric ancient pancake as an inner barrier: consider, for each $r>0$, the set
\begin{equation}\label{eq:inner barrier}
P^R_r\doteqdot \bigcap_{f\in F^{n-1}_P}\big\{p\in \{0\}\times \R^n:\inner{p}{z_f}<R-r\big\}\,.
\end{equation}

\begin{figure}[h]
\centering
\includegraphics[width=0.6\textwidth]{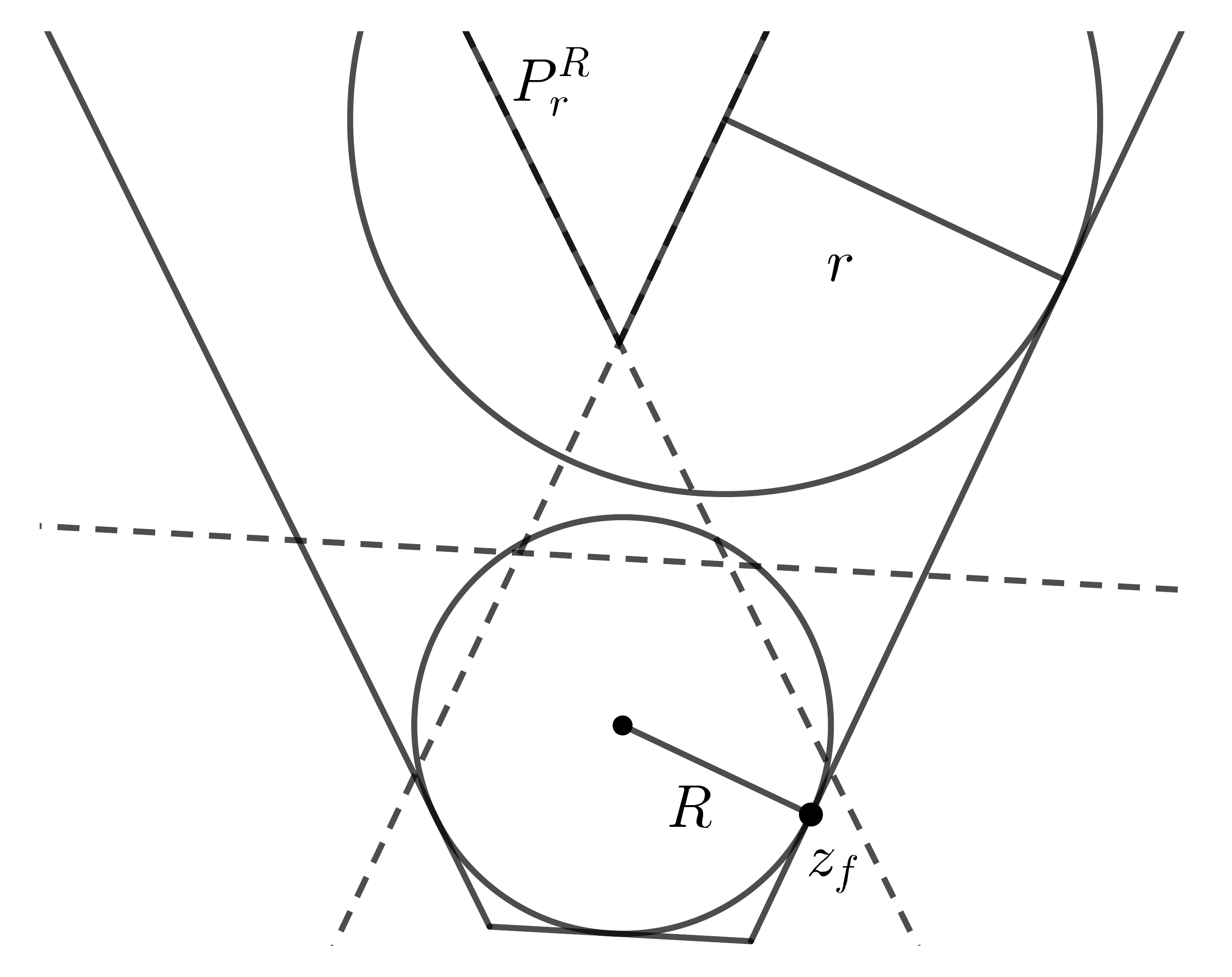}
\caption{The ``barrier set'' $P^R_r$ above the mid-section of the initial datum $\M^R$.}
\end{figure}


By \eqref{eq:pancake edge asymptotics}, we can choose $t_r=-r+o(r)$ such that, for each $p\in P^R_r$, the translated time $t_r$ slice of the rotationally symmetric ancient pancake, $\Pi_{t_r}-p$, lies to the inside of $\M^R$ for $r$ sufficiently large. It follows that $P^R_{r}$ lies to the inside of $\M^R_{t}$ until time at least $t=-t_r$. Since $\M^R_t$ lies between the two parallel hyperplanes $\{\pm\frac{\pi}{2}\}\times \R^n$ and since $G(\M^R_t)$ is a convex subset of the sphere, we conclude that $G(\M^R_t)$ is constant for $t\in(0,-t_r)$. The claim follows since $t_r\to-\infty$ as $r\to\infty$.

In the semi-degenerate case, a different argument is required. We make use of the fact that the old-but-not-ancient solutions are limits of compact solutions: let $\{\M^{R,L}_t\}_{t\in[0,T_{R,L})}$ be the compact approximators introduced in the proof of Lemma \ref{lem:old-but-not-ancient}. Given $z\in \{0\}\times S^{n-1}$, the evolution equation \eqref{eq:MCF support} for the support function and the differential Harnack inequality \eqref{eq:Harnack nonancient} yields
\begin{align}
\sigma_{R,L}(z,t)={}&\sigma_{R,L}(z,0)-\int_{0}^{t}H_{R,L}(z,s)ds\nonumber\\
\ge{}&\sigma_{R,L}(z,0)-\int_{0}^{t}\sqrt{\frac{t}{s}}H_{R,L}(z,t)ds\nonumber\\
={}&\sigma_{R,L}(z,0)-2tH_{R,L}(z,t)\,.\label{eq:Gauss image semi-degenerate}
\end{align}
If there exists $z\in G(\M^R_t)\setminus G(P)$ for some $t>0$, then, taking $L\to\infty$ in \eqref{eq:Gauss image semi-degenerate}, we find that $\sigma_{R,L}(z,t)\to\infty$ (since $\s_{R, L}(z,0)\to \infty$ and $H_{R, L}(z,t)$ is bounded uniformly in $L$), which is impossible since $\sigma_R(z,t)$ is finite. The claim follows.
\end{proof}

The following lemma implies, in particular, that the ``vertices'' of the solution remain the correct distance outside of the rotationally symmetric example.

\begin{lemma}\label{lem:speed lower bound}
For each vertex $v\in F^0_P$,
\begin{equation}\label{eq:lower speed bound}
H_R(z,t)-\left\langle z,v\right\rangle\ge 0\,,
\end{equation}
and hence
\begin{equation}\label{eq:sigma ratio bound}
\frac{\sigma_R(z,t)-\sigma_R(z,s)}{t-s}\le -\inner{z}{v}
\end{equation}
for all $z\in G(\M^R_t)$ and distinct $s,t\in (0,T_R)$. 
\end{lemma}
\begin{proof}
In the compact case, the first claim follows immediately from the maximum principle. Indeed, working in the standard parametrization we obtain, for each $v\in F^0_P$,
\[
\liminf_{t\to 0}\big(H_R(\cdot,t)-\left\langle \nu_R(\cdot,t),v\right\rangle\!\big)\ge 0
\]
(by Lemma \ref{lem:H initial lower bound}) and
\[
(\partial_t-\Delta_R-\vert{\operatorname{II}_R}\vert^2)\big(H_R-\left\langle\nu_R,v\right\rangle\!\big)=0\,,
\]
where $\nu_R$ is the outward unit normal field, ${\operatorname{II}_R}$ is the second fundamental form, and $\Delta_R$ the induced Laplacian.

In the noncompact case, we first obtain the estimate on the (compact) solutions arising from the ``doubled'' initial data $\M^{R,L}$ (see the proof of Lemma \ref{lem:old-but-not-ancient}) and then take a limit as $L\to\infty$. (Note that the initial estimates still hold on $\M^{R,L}$ since the mean curvature is unchanged under reflection, whereas the reflected normals point \emph{away} from the original vertices).

Since, by Lemma \ref{lem:Gauss image}, the Gauss image of $\M^R_t$ is constant in $t$, \eqref{eq:sigma ratio bound} follows by integrating the evolution equation \eqref{eq:MCF support} for $\sigma_R$ and applying the speed bound \eqref{eq:lower speed bound}.
%
\end{proof}

Observe that Lemma \ref{lem:speed lower bound} implies in particular that, for $t<t_0$,
\begin{equation}\label{eq:lower sigma bound}
\frac{\sigma_R(z,t)-\sigma_R(z,t_0)}{t_0-t}\ge \max_{v\in F^0_P}\inner{z}{v}=\sigma[P](z)
\end{equation}
for all $z\in G(\M^R_t)$, where $\sigma[P]$ denotes the support function of $P$.

In order to obtain a (subsequential) limit as $R\to\infty$, we need a uniform (time-dependent) lower bound for the inradius. We achieve this in two steps (Lemmas \ref{lem:minimum speed estimate} and \ref{lem:width_est}). 
The first step is a bound for the speed away from the ``edge'' region.

\begin{lemma}\label{lem:minimum speed estimate}
For every $h\in(0,1)$ there exist $C_h<\infty$ and $r_h<\infty$ (depending only on $n$ and $h$) such that
\[
\min_{\M^R_t\cap (\R\times B_{C_h}^n(p))}H_R\le C_h\mathrm{e}^{-h^2r}
\]
for all $p\in P_t^r$ and $r\ge r_h$, where
\[
P_t^r\doteqdot \bigcap_{z\in G(\M^R_t)}\{p\in\{0\}\times\R^n:\inner{p}{z}\le \sigma(z,t)-r\}\,.
\]
\end{lemma}
\begin{proof}
Fix $h\in(0,1)$ and consider, for each $t<0$, the hypersurface $\Pi^h_t$ obtained by rotating the scaled Angenent oval $h^{-1}\mathrm{A}_{h^2t}$ about the $x$-axis. Namely (see \S \ref{ssec:rotational pancake}),
\[
\Pi^h_t=\{x(\theta,t)e_1+y(\theta,t)\phi:(\theta,\phi)\in [-\tfrac{\pi}{2},\tfrac{\pi}{2}]\times S^{n-1}\},
\]
where, setting $a_h(t)\doteqdot (e^{-2h^2t}-1)^{-\frac{1}{2}}$,
\[
x(\theta,t)\doteqdot \frac{1}{h}\arctan\left(\frac{\sin\theta}{\sqrt{\cos^2\theta+a_h^2(t)}}\right)
\]
and
\begin{equation}\label{eq:AO y scaled}
y(\theta,t)\doteqdot -ht+\frac{1}{h}\log\left(\frac{\sqrt{\cos^2\theta+a_h^2(t)}+\cos\theta}{\sqrt{1+a_h^2(t)}}\right)
\end{equation}
are the coordinates of $h^{-1}\mathrm{A}_{h^2t}$ with respect to its clockwise oriented turning angle $\theta$ as in  \S \ref{ssec:rotational pancake}.

Since $h<1$, \eqref{eq:AO y scaled} implies that $\{x=0\}\cap (\Pi^h_{-r}-p)$ lies to the inside of $\{x=0\}\cap \M^R_{t}$ when $p\in P^r_t$ and $r\ge r_h$, say. Moreover, since $h<1$, the hypersurface $\Pi^h_{-r}$ must eventually leave the slab of width $\pi$ as $r$ increases. Thus, choosing $r_h$ larger if necessary, we conclude that the two hypersurfaces must intersect each other for all $r\ge r_h$. Now translate $\Pi^h_{-r}$ along the $x$-axis towards the negative direction until it touches $\M^R_t$ from the inside at a point $p=x(\theta,-r)e_1+y(\theta,-r)\phi$ with $x(\theta,-r)>0$. By \eqref{eq:Angenent oval kappa}, we have (in a limiting sense if $\theta=\frac{\pi}{2}$)
\begin{align*}
\nu_R={}&\nu_{\Pi^h_{-r}}=\sin\theta e_1+\cos\theta \phi\\
 H_R\le{}&H_{\Pi^h_{-r}}=h\sqrt{\cos^2\theta+a_h^2(-r)}+\cos\theta\frac{n-1}{y(\theta,-r)}
\end{align*}
at $p$. Choose $v\in F^0_P$ so that $\inner{v}{\phi}\ge 1$. By \eqref{eq:lower speed bound},
\begin{align*}
\cos\theta\inner{\phi}{v}={}&\inner{\nu_{\Pi^h_{-r}}(p)}{v}\\
={}&\inner{\nu_{R}(p)}{v}\\
\le{}& H_R(p)\\
\le{}& H_{\Pi^h_{-r}}(p)=\cos\theta\left(h\sqrt{1+\frac{a_h^2(-r)}{\cos^2\theta}}+\frac{n-1}{y(\theta,-r)}\right)
\end{align*}
and hence (if $\theta\neq \frac{\pi}{2}$)
\[
1\le h\sqrt{1+\frac{a_h^2(-r)}{\cos^2\theta}}+\frac{n-1}{y(\theta,-r)}\,.
\]
It follows (trivially if $\theta=\frac{\pi}{2}$) that
\begin{equation}\label{eq:cos theta small}
\cos\theta\le C_h\mathrm{e}^{-h^2r}
\end{equation}
for some $C_h<\infty$ for all $r\ge r_{h}$, say. Indeed, if $\cos\theta\ge C\mathrm{e}^{-h^2r}$, then
\begin{align*}
1\le {}& h\sqrt{1+\frac{a_h^2(-r)}{C^2\mathrm{e}^{-2h^2r}}}+\frac{(n-1)h}{\log\left(\mathrm{e}^{h^2r}\frac{\sqrt{C^2\mathrm{e}^{-2h^2r}+a_h^2(-r)}+C\mathrm{e}^{-h^2r}}{\sqrt{1+a_h^2(-r)}}\right)}\\
={}&h\sqrt{1+\frac{a_h^2(-r)}{C^2\mathrm{e}^{-2h^2r}}}+\frac{(n-1)h}{\log\left(\frac{\sqrt{C^2+\mathrm{e}^{2h^2r}a_h^2(-r)}+C}{\sqrt{1+a_h^2(-r)}}\right)}\,.
\end{align*}
Since $\mathrm{e}^{2h^2r}a_h^2(-r)\to 1$ as $r\to\infty$, we can find $C=C_h<\infty$ so that the right hand side is less than 1 for all $r\ge r_h$, which proves the claim.

Since the mean curvature of $\Pi^h_{-r}$ is monotone decreasing in $\theta$ for $\theta\in (0,\frac{\pi}{2})$, we find that
\[
H_R(p)\leq C_he^{-h^2r}
\]
for $r\ge r_h$. Finally, we observe that
\[
y(\theta,-r)\le \frac{1}{h}\log\left(\frac{\sqrt{C_h^2+\mathrm{e}^{2h^2r}a_h^2(-r)}+C_h}{\sqrt{1+a_h^2(-r)}}\right)
\]
is bounded uniformly for $r\ge r_h$.
\end{proof}


An elliptic version of the following estimate was exploited in \cite{SX} (and modified in \cite{BLT2}) in order to construct translators in slabs in $\R^{n+1}$.

\begin{lemma}\label{lem:width_est}
Each old-but-not-ancient solution $\{\M^R_t\}_{t\in (0,T_R)}$ satisfies
\[
\vert x\vert\geq \frac{\pi}{2}\left(1-H_R\right),
\]
where $x(p)\doteqdot \inner{p}{e_1}$.
\end{lemma}
\begin{proof}
Suppose first that $P$ is bounded. Define
\[
v\doteqdot 1-\frac{2}{\pi}x\,.
\]
Then, along $\{\M^R_t\}_{t\in(0,T_R)}$,
\[
(\partial_t-\Delta_R)v=0
\]
and hence, by the maximum principle,
\begin{align*}
\inf_{\M^R_t\cap \{x>0\}}\frac{H_R}{v}\ge{}& \min\left\{\liminf_{t\to 0}\inf_{\M^R_t\cap \{x>0\}}\frac{H_R}{v},\inf_{t\in(0,T_R)}\min_{\M^R_t\cap \{x=0\}}H_R\right\}.
\end{align*}
Observe that
\begin{align*}
\liminf_{t\to 0}\inf_{\M^R_t\cap \{x>0\}}\frac{H_R}{v}\ge {}&\inf_{x\in(0,\frac{\pi}{2})}\frac{\cos x}{1-\frac{2}{\pi}x}=1
\end{align*}
for all $R>0$. On the other hand, by \eqref{eq:lower speed bound},
\[
\min_{\M^R_t\cap \{x=0\}}H_R\ge\min_{z\in \{0\}\times S^{n-1}}\max_{v\in F_P^0}\left\langle z,v\right\rangle =1
\]
for all $t>0$. This proves the claim.

If $P$ is unbounded, we apply the same argument to the ``doubled'' approximating solutions, $\{\M^{R,L}\}_{t\in[0,T_{R,L})}$. This yields
\begin{align*}
\inf_{\M^{R,L}_t\cap \{x>0\}}\!\!\!\frac{H_{R,L}}{v}\ge{}& \min\!\left\{\liminf_{t\to 0}\!\!\!\inf_{\M^{R,L}_t\cap \{x>0\}}\!\!\!\frac{H_{R,L}}{v},\!\inf_{t\in(0,T_{R,L})}\!\min_{\M^{R,L}_t\cap \{x=0\}}\!\!H_{R,L}\right\}\!.
\end{align*}
Again,
\begin{align*}
\liminf_{t\to 0}\inf_{\M^{R,L}_t\cap \{x>0\}}\frac{H_{R,L}}{v}\ge {}&\inf_{x\in(0,\frac{\pi}{2})}\frac{\cos x}{1-\frac{2}{\pi}x}=1
\end{align*}
by construction of the initial datum. Since, by the argument leading to \eqref{eq:lower speed bound},
\[
H_{R,L}(z,t)\ge \max_{v\in F_P^0}\inner{z}{v}
\]
for all $(z,t)\in S^n\times(0,T_{R,L})$ and all $L>0$, 
we conclude that
\[
\frac{H_{R}}{v}(z,t)\ge \min\left\{1,\max_{v\in F_P^0}\inner{z}{v}\right\}\ge 1\,.\qedhere
\]
\end{proof}

Since the solution is convex, Lemmas \ref{lem:minimum speed estimate} and \ref{lem:width_est} yield the uniform lower width bound
\begin{equation}\label{eq:width estimate}
\vert x(p)\vert\ge \frac{\pi}{2}-C_h\mathrm{e}^{-h^2r}
\end{equation}
for all
\[
p\in \bigcap_{z\in G(\M^R_t)}\{q\in\{0\}\times\R^n:\inner{q}{z}\le \sigma(z,t)-r-C_h\}
\]
and $r\ge r_h$. In particular, this implies a uniform lower bound for the inradius.

Next, we study the asymptotic behaviour of the unbounded examples when $t\to+\infty$. Given a circumscribed convex body
\[
P= \bigcap_{z\in F}\{p\in \R^n:\inner{p}{z}\le 1\},
\]
where $F$ is a subset of $S^{n-1}$, define the \textsc{exscribed body} $P_{-1}$ by
\[
P_{-1}\doteqdot \bigcap_{z\in F}\{p\in\R^n:\inner{p}{z}\le -1\}.
\]

\begin{lemma}\label{lem:squash-up}
If $P\in {\mathrm{P}}{}_\ast^n$ is unbounded, but nondegenerate, then
\[
\lim_{t\to\infty}\frac{1}{t}\Omega^R_t=\bigcap_{f\in F^{n-1}_P}\big\{p\in\R^n:\inner{p}{z_f}\le -1\big\}.
\]
\end{lemma}
We shall refer to the limit as the \textsc{forward squash-down} of the immortal solution $\{\M^R_t\}_{t\in[0,\infty)}$.
\begin{proof}[Proof of Lemma \ref{lem:squash-up}]
By \eqref{eq:sigma ratio bound}, 
\begin{equation}\label{eq:forward squash-down lower bound}
\limsup_{t\to\infty}\frac{\sigma_R(z,t)}{t}\le -\max_{v\in F^0_P}\inner{z}{v}.
\end{equation}

To obtain a lower bound, consider, as in the proof of Lemma \ref{lem:Gauss image}, the set
\[
P_r^R\doteqdot\bigcap_{f\in F^{n-1}_P}\{p\in\{0\}\times \R^{n}:\inner{p}{z_f}<R-r\}\,.
\]
By \eqref{eq:ancient pancake radius} and \eqref{eq:pancake edge asymptotics}, for $r$ sufficiently large we can find $t_r=-r+o(r)$ such that, for each $p\in P_r^R$, the translated time $t_r$ slice of the rotationally symmetric ancient pancake, $\Pi_{t_r}-p$, lies to the inside of the initial hypersurface $\M^R$. But then, by the avoidance principle,
\[
P^R_r\subset \Omega^R_{-t_r}\,,
\]
where $\Omega^R_{-t_r}$ is the convex body bounded by $\M^R_{-t_r}$. Thus,
\[
\frac{1}{-t_r}P^R_r\subset \frac{1}{-t_r}\Omega^R_{-t_r}.
\]
The claim follows since $t_r=-r+o(r)$ as $r\to\infty$.
\end{proof}

The next estimate implies that at least one ``facet'' of $\{\M^R_t\}_{t\in[0,T_R)}$ remains close to the rotationally symmetric example. In case $P$ is regular, this implies that \emph{all} of the facets remain close.


\begin{lemma}\label{lem:horizontal displacement estimate}
For every $h>1$ there exists $t_h>-\infty$ such that, given any $t_0\in (0,T_R)$ and $p_0\in \M^R_{t_0}\cap(\{0\}\times\R^n)$,
\[
\min_{z\in G(\M^R_t)\cap (\{0\}\times\R^n)}\left(\sigma_R(z,t)-\inner{p_0}{z}\right)\le h(t_0-t)
\]
for all $t\in(0,t_0)\cap(-\infty,t_0+t_h)$.
\end{lemma}
\begin{proof}
By the width estimate \eqref{eq:width estimate} and the displacement estimate \eqref{eq:ancient pancake radius}, given $h_\ast>1$ and $h\in(1,h_\ast)$, there exists $t_h>-\infty$ such that if
\[
\min_{z\in G(\M_t)\cap (\{0\}\times\R^n)}\left(\sigma_R(z,t)-\inner{p_0}{z}\right)> h_\ast(t_0-t)
\]
and $t_0-t>-t_h$, then the time $t$ slice $\Pi^h_t=h^{-1}(\Pi_{h^2(t-t_0)}-p_0)$ of the scaled-by-$h$ rotationally symmetric ancient pancake centred at $(p_0,t_0)$ lies to the inside of $\M^R_t$. But this violates the avoidance principle, since both $\{\M^R_t\}_{t\in[0,T_R)}$ and $\{\Pi^h_t\}_{t\in(-\infty,t_0)}$ reach the point $p_0$ at time $t=t_0$.
\end{proof}

\section{Ancient/translating solutions with prescribed squash-downs}\label{sec:existence}

After time-translating the old-but-not-ancient solutions constructed in Lemma \ref{lem:old-but-not-ancient} so that they reach the origin at time zero, Lemma \ref{lem:horizontal displacement estimate} allows us to control the displacement of the faces, and thereby obtain ancient solutions.

\subsection{Regular examples} 

The construction is particularly straightforward when $P$ is a \emph{regular} polytope. We first consider the case that $P$ is bounded.

\begin{theorem}\label{thm:bounded polygonal pancakes}
Let $P\in \mathrm{P}{}^n_\ast$ be a bounded regular polytope. There exists a compact, convex, locally uniformly convex ancient solution $\{\M_{t}\}_{t\in(-\infty,0)}$ to mean curvature flow which sweeps out $(-\frac{\pi}{2},\frac{\pi}{2})\times \R^n$, contracts to the origin as $t\to 0$, and whose squash-down is $P$. It is reflection symmetric across the hyperplane $\{0\}\times\R^n$ and inherits the symmetries of $P$.
\end{theorem}
\begin{proof}
For each $R>0$, let $\{\M^R_t\}_{t\in(\alpha_R,0)}$ be the old-but-not-ancient solution obtained from time translating the solution constructed in Lemma \ref{lem:old-but-not-ancient} to that it reaches the origin at time zero. Since, by Lemmas \ref{lem:minimum speed estimate} and \ref{lem:width_est}, the inradius of each $\M^R_t$ is bounded from below uniformly in $R$, the Blaschke selection theorem and the interior estimates of Ecker and Huisken \cite{EckerHuisken91} yield a sequence of scales $R_j\to\infty$ such that the sequence of flows $\{\M^{R_j}_t\}_{t\in[\alpha_{R_j},\,\omega_{R_j})}$ converges locally uniformly in the smooth topology to a family $\{\M_t\}_{t\in(-\infty,\,\omega)}$ of smooth, convex hypersurfaces $\M_t$ which lie in the slab $(-\frac{\pi}{2},\frac{\pi}{2})\times\R^n$ and evolve by mean curvature flow. Without loss of generality, we may assume that $\omega$ is the maximal time. 

By construction, $\{\M_t\}_{t\in(-\infty,\,\omega)}$ is reflection symmetric across the hyperplane $\{0\}\times\R^n$, inherits the symmetry group of $P$, and  lies in $(-\frac{\pi}{2},\frac{\pi}{2})\times\R^n$. By the width estimate \eqref{eq:width estimate}, it lies in no smaller slab.

Since $P$ is bounded, Lemma \ref{lem:horizontal displacement estimate} implies that $\M_t$ is bounded for each $t$. In particular, $\omega<\infty$ and the solution contracts to a point as $t\to\omega$. By the symmetries inherited by the limit, the final point must be the origin. After translating in time, we may arrange that $\omega=0$.

By \eqref{eq:lower sigma bound} and Lemma \ref{lem:horizontal displacement estimate},
\begin{equation}\label{eq:short axis estimate}
\min_{z\in G(\M_t)\cap(\{0\}\times\R^n)}\frac{\sigma(z,t)}{-t}\to 1\;\;\text{as}\;\; t\to\infty\,.
\end{equation}
Together, \eqref{eq:lower sigma bound}, \eqref{eq:short axis estimate} and Lemma \ref{lem:Gauss image} imply that the squash-down $\Omega_\ast$ of the limit $\{\M_t\}_{t\in(-\infty,0)}$ is $P$. 

The splitting theorem for the second fundamental form implies that the time-slices $\M_t$ of $\{\M_t\}_{t\in(-\infty,0)}$ are locally uniformly convex.
\end{proof}

Next, we consider unbounded $P$.

\begin{theorem}\label{thm:unbounded polygonal pancakes}
Let $P\in \mathrm{P}{}^n_\ast$ be an unbounded regular polytope which does not split off a line. There exists a convex, locally uniformly convex translating solution $\{\M_{t}\}_{t\in(-\infty,\infty)}$ to mean curvature flow which sweeps out $(-\frac{\pi}{2},\frac{\pi}{2})\times \R^n$, has its tip at the origin at time zero, and whose squash-down is $P$. It is reflection symmetric across the hyperplane $\{0\}\times\R^n$ and inherits the symmetries of $P$.
\end{theorem}
\begin{proof}
For each $R>0$, let $\{\M^R_t\}_{t\in(\alpha_R,\infty)}$ be the old-but-not-ancient solution obtained from time translating the solution constructed in Lemma \ref{lem:old-but-not-ancient} to that it reaches the origin at time zero. 

Since, by Lemmas \ref{lem:minimum speed estimate} and \ref{lem:width_est}, the inradius of each $\M^R_t$ is bounded from below uniformly in $R$, the Blaschke selection theorem and the interior estimates of Ecker and Huisken \cite{EckerHuisken91} yield a sequence of scales $R_j\to\infty$ such that the sequence of flows $\{\M^{R_j}_t\}_{t\in[\alpha_{R_j},\infty)}$ converges locally uniformly in the smooth topology to a family $\{\M_t\}_{t\in(-\infty,\infty)}$ of smooth, convex hypersurfaces $\M_t$ which lie in the slab $(-\frac{\pi}{2},\frac{\pi}{2})\times\R^n$ and evolve by mean curvature flow.

By construction, $\{\M_t\}_{t\in(-\infty,\infty)}$ is reflection symmetric across the hyperplane $\{0\}\times\R^n$, inherits the symmetry group of $P$, and lies in $(-\frac{\pi}{2},\frac{\pi}{2})\times\R^n$. By the width estimate \eqref{eq:width estimate}, it lies in no smaller slab.

By \eqref{eq:lower sigma bound} and Lemma \ref{lem:horizontal displacement estimate},
\begin{equation}\label{eq:short axis estimate nc}
\min_{z\in G(\M_t)\cap(\{0\}\times\R^n)}\frac{\sigma(z,t)}{-t}\to 1\;\;\text{as}\;\; t\to\infty\,.
\end{equation}
Together, \eqref{eq:lower sigma bound}, \eqref{eq:short axis estimate nc} and Lemma \ref{lem:Gauss image} imply that the squash-down $\Omega_\ast$ of the limit $\{\M_t\}_{t\in(-\infty,\infty)}$ is $P$.

Since $P$ does not split off a line, the splitting theorem for the second fundamental form implies that the time-slices $\M_t$ of $\{\M_t\}_{t\in(-\infty,\infty)}$ are locally uniformly convex.

It remains to show that the limit is a translator. Set $e\doteqdot v/\vert v\vert$, where $v$ is the vertex of $P$. By \eqref{eq:forward backward limits} and Lemma~\ref{lem:squash-up},
\[
\lim_{t\to-\infty}H(e,t)=\sigma_\ast(e)=-\sigma^\ast(e)=\lim_{t\to\infty}H(e,t)\,.
\]
Since $\{\M_t\}_{t\in(-\infty,\infty)}$ is a limit of solutions which satisfy the differential Harnack inequality (themselves being limits of compact solutions), it too satisfies the differential Harnack inequality. We conclude that $H(e,t)$ is constant in $t$, at which point the conclusion follows from the rigidity case of the differential Harnack inequality.
\end{proof}

\subsection{Irregular examples}

Next, we obtain examples out of any circumscribed (bounded or unbounded) simplex. (Recall that an unbounded simplex is just a cone whose link is a bounded simplex).

\begin{theorem}\label{thm:irregular polygonal pancakes bounded}
For each circumscribed bounded simplex $P\in \mathrm{P}{}^n_\ast$ there exists a compact, convex ancient solution $\{\M_{t}\}_{t\in(-\infty,0)}$ to mean curvature flow which sweeps out $(-\frac{\pi}{2},\frac{\pi}{2})\times \R^n$, contracts to the origin at time zero and whose squash-down is $P$. It is reflection symmetric across the hyperplane $\{0\}\times\R^n$.
\end{theorem}
\begin{proof}
Consider the old-but-not ancient solution $\{\M^R_t\}_{t\in[\alpha_R,0)}$ obtained by spacetime translating the approximating solution constructed in Lemma \ref{lem:old-but-not-ancient} so that it contracts to the origin at time zero. Proceeding as in Theorem \ref{thm:bounded polygonal pancakes}, we obtain a limit ancient solution $\{\M_t\}_{t\in(-\infty,\,\omega)}$ along a sequence of scales $R_j\to\infty$. A priori, in taking the limit, we could ``lose'' facets. However, this can be prevented using Lemma \ref{lem:speed lower bound} and a pancake barrier argument (cf. Lemma \ref{lem:horizontal displacement estimate}).

Indeed, observe first that \eqref{eq:sigma ratio bound} implies that the old-but-not-ancient solutions satisfy
\[
\frac{\sigma_R(z,t)}{-t}\ge \max_{v\in F^0_P}\inner{z}{v}\,.
\]
This implies that the squash-down $\Omega_\ast$ of $\{\M_t\}_{t\in(-\infty,\infty)}$ contains $P$. 

Suppose, then, that $\Omega_\ast\not\subset P$. Then, since $P$ is a simplex, there exist $p_\ast\in \Omega_\ast$ and $h_\ast>1$ such that $S^{n-1}_{h_\ast}(p_\ast)\subset \Omega_\ast$. By the differential Harnack inequality, $\Omega_\ast\subset \frac{1}{-t}\Omega_t$. Thus, by the width estimate \eqref{eq:width estimate} and the displacement estimate \eqref{eq:ancient pancake radius}, for any $h\in (1, h_\ast)$ we can find $t_h<0$ such that, for each $t<t_h$, the translated time $t$ slice $\Pi^h_t=h^{-1}\Pi_{h^2t}+tp_\ast$ of the scaled-by-$h$ rotationally symmetric ancient pancake centred at $-tp_\ast$ lies to the inside of $\M^{R_j}_t$ for $j$ sufficiently large. But this violates the avoidance principle, since the pancake contracts to the point $-t p_\ast$ after time $-t$, whereas the old-but-not ancient solutions contract to the origin after this time.
\end{proof}

\begin{theorem}\label{thm:irregular polygonal pancakes unbounded}
For each circumscribed unbounded simplex $P\in \mathrm{P}{}^n_\ast$ there exists a translating solution $\{\M_{t}\}_{t\in(-\infty,\infty)}$ to mean curvature flow which sweeps out $(-\frac{\pi}{2},\frac{\pi}{2})\times \R^n$ and whose squash-down is $P$. It is reflection symmetric across the hyperplane $\{0\}\times\R^n$.
\end{theorem}
\begin{proof}

Consider the old-but-not ancient solution $\{\M^R_t\}_{t\in[\alpha_R,\infty)}$ obtained by translating in time so that the final tangent hyperplane to reach the origin does so at time zero, and then translating in space so that this hyperplane supports the solution at the origin. Proceeding as in Theorem \ref{thm:unbounded polygonal pancakes}, we obtain as limit ancient solution $\{\M_t\}_{t\in(-\infty,\omega)}$ along a sequence of scales $R_j\to\infty$. A priori, in taking the limit, we could ``lose'' facets. However, this can be prevented using Lemma \ref{lem:speed lower bound} and a pancake barrier argument (cf. Lemma \ref{lem:horizontal displacement estimate}) similar to the one used in the proof of Theorem \ref{thm:irregular polygonal pancakes bounded}. (Note that in this case the point $p_\ast$ can be chosen so that $-tp_\ast$ lies outside the time $t=0$ slice of the solution.)
\end{proof}

\begin{remark}
By assuming some symmetries, the argument of Theorems \ref{thm:irregular polygonal pancakes bounded} and \ref{thm:irregular polygonal pancakes unbounded} can be used to construct many more examples (consider e.g. circumscribed kites). This strongly suggests the existence of examples with \emph{any} given (irregular) circumscribed squash-down. However, the removal of symmetry assumptions in general would require a refinement of the pancake barrier argument (in order to prevent the squash-downs from degenerating to known examples).
\end{remark}

Our next proposition provides a refinement of the asymptotic behaviour as $t\to-\infty$. In particular, it verifies that the solutions decompose into the ``correct'' translating solutions, at least at the level of the squash-down.


\begin{proposition}\label{prop:asymptotics}
Given  $P\in {\mathrm{P}}{}_\ast^n$, a regular polytope or a circumscribed simplex (bounded or unbounded), let $\{\M_{t}\}_{t\in(-\infty,\,\omega)}$ be the ancient solution constructed in one of the Theorems~\ref{thm:bounded polygonal pancakes}, or  \ref{thm:irregular polygonal pancakes bounded}  and with squash-down $P$. Given a face $V$ of $P$ and $z\in \relint G(V)$, where $G(V)$ is the set of unit outward normals to hyperplanes that support $P$ and contain $V$, the squash-down of any asymptotic translator corresponding to $z$ is the support cone of $P$ at $V$. 
\end{proposition}


\begin{remark}
Although the proposition is intuitively clear, it is far from immediate: roughly speaking, two hyperplanes with normals $z$ and $w$ will, in general, support the same face of the squash-down only if $\vert X(z,t)-X(w,t)\vert\le o(-t)$ as $t\to-\infty$. However, in general, they will only correspond to the same asymptotic translator if $\vert X(z,t)-X(w,t)\vert\le O(1)$ as $t\to-\infty$.
\end{remark}

\begin{remark}\label{rem:n=2 decomposition}
In Section \ref{sec:squash-down}, we will see that an asymptotic translator splits off a line if and only if its squash-down splits off the same line (Corollary \ref{cor:translators}). This implies that the $(n-1)$-faces correspond to Grim Reapers and the $(n-2)$-faces correspond to flying wings, since (see Corollary \ref{cor:Grim uniqueness} and Remark \ref{rem:wing uniqueness}) both are uniquely determined by their squash-downs (it is not known whether or not this is always the case). In particular, when $n=2$, this gives a canonical decomposition of the ancient solution into asymptotic translators: a finite number of Grim planes (corresponding to edges) soldered together by flying wings (corresponding to vertices).
\end{remark}

\begin{remark}\label{rem:translator asymptotic translators}
In Section \ref{sec:squash-down}, we will obtain a similar asymptotic description in regions corresponding to ``asymptotic normals'' in the noncompact case (see Proposition~\ref{prop:converging sequences} below).
\end{remark}

\begin{proof}[Proof of Proposition \ref{prop:asymptotics}]
Let $V\in F^k_P$, be a $k$-face of $P$, $0\le k\le n$. Let  $z_V\in G(V)$ the maximum of $\s_\ast$ in $G(V)$ and note that
\begin{equation}\label{eq:vertex of V}
\s_\ast(w)=\s_\ast(z_V)\langle w, z_V\rangle\;\;\text{for all}\;\; w\in G(V)\,.
\end{equation}
It suffices to show that 
\begin{equation}\label{transest}
G(\Sigma_V)=\relint G(V)
\end{equation}
for any asymptotic translator corresponding to $z_V$ and that, for any $z\in\relint G(V)$, $z$ is in the Gauss image of some asymptotic translator corresponding to $z_V$. We will achieve this by induction on $k$.

Note first that we may exclude the case $k=n$, since then $z=\pm e_1$ and $\Sigma_z$ is the pair of parallel hyperplanes $\{0\}\times \R^n\cup \{\mp \pi\}\times \R^n$, which implies the proposition. 
\begin{claim}\label{k=0}
If $V$ is a $0$-face, then $G(\Sigma_V)= \relint G(V)$ for any translator $\Sigma_V$ corresponding to $z_V$, and any $z\in\relint G(V)$ is in the Gauss image of some asymptotic translator corresponding to $z_V$.
\end{claim}
\begin{proof}
Given $z\in \relint G(V)$, let $\Sigma_z= \lim_{i\to\infty}(\M_{t_i}-X(z^i, t_i))$, for a sequence of normals $z^i\to z$ and a sequence of times $t_i\to-\infty$, be an asymptotic translator corresponding to $z\in\relint G(V)$ (obtained as in Lemma \ref{lem:backwards asymptotic translators}).
First note that
\begin{equation}\label{inclusion}
G(\Sigma_z)\subset G(V)\,.
\end{equation}
Indeed, if not, then we can find $z_0\in G(\Sigma_z)\setminus G(V)$. Since $z_0\in G(\Sigma_z)$, there exists a sequence of normals $z_i\in G_\ast$ with $z_i\to z_0$  such that $|X(z_i, t_i)- X(z^i, t_i)|$ is uniformly bounded. In particular (cf. \eqref{eq:limit exists} in Lemma \ref{lem:backwards asymptotic translators}),
\begin{equation}\label{close}
(-t_i)^{-1}|X(z_i, t_i)- X(z^i, t_i)|\to 0\,\,\text{ as } i\to \infty\,.
\end{equation}
Let $V_0\in F^\ell_P$, $0\le \ell\le n$, be the $\ell$-face of $P$ for which $z_0\in \relint G(V_0)$. By definition of the squash-down, \eqref{close} can only happen if $z\in G(V_0)$, and so $\ell\le k-1$, which is impossible. This proves \eqref{inclusion}.

We next show that $G(\Sigma_V)= \relint G(V)$, for any translator $\Sigma_V$ corresponding to $z_V$. Recall from Proposition \ref{prop:translator squash-down} that the squash-down of $\Sigma_V$ is the cone 
\[
(\Sigma_V)_\ast=\bigcap_{z\in G(\Sigma_V)}\{p\in\{0\}\times \R^n:\inner{p}{z}\le -\inner{z}{\vec v}\}\,.
\]
We claim that $(\Sigma_V)_\ast$ is circumscribed. Indeed, if not, then we can find an asymptotic translator for $\Sigma_V$ whose squash-down is not circumscribed. Continuing in this manner, we eventually obtain a Grim hyperplane whose squash-down is not circumscribed. But this violates the width estimate (Lemma \ref{lem:width_est}). Alternatively, we may simply apply Theorem \ref{thm:circumscribed} below.

The structure of $P$ and \eqref{inclusion} now imply that $G(\Sigma_V)=\relint G(V)$. Indeed if $P$ is regular, then, by its symmetries, $\ov{G(\Sigma_V)}$ must contain all $z\in G(V)$ such that $\s_\ast(z)=1$ and thus $G(\Sigma_V)\supseteq \relint G(V)$. If $P$ is a simplex, then the claim follows because if $G(\Sigma_z)\subsetneq \relint G(V)$ then the squash-down cannot be a circumscribed cone.

Finally, we show that any $z$ is in the Gauss image of some asymptotic translator corresponding to $z_V$. To do this, it suffices to show that $z_V\in G(\Sigma_z)$. Indeed,  if that is the case, then there exists a sequence of normals $z_V^i\to z_V$ such that $\limsup_{i\to\infty}|X(z_V^i, t_i)- X(z^i, t_i)|<\infty$, which implies that $z$ is in the Gauss image of an asymptotic translator corresponding to $z_V$. We will show that the bulk velocity of $\Sigma_z$ is given by 
 \begin{equation}\label{eq:tip2}
\vec v=-\s_\ast(z_V)z_V\,.
\end{equation}
Let $w_0\in G(\Sigma_z)$ be such that
\begin{equation}\label{max}
w_0=\frac{\vec v}{|\vec v|}\,.
\end{equation}
(The existence of such $w_0$ is guaranteed by the structure of $P$, or by part \eqref{item:tip attained} of Corollary \ref{cor:translators} below). For any $w\in G(\Sigma_z)$, let $H(w)$ denote the mean curvature of $\Sigma_z$ at the point with normal $w$. Then, by Lemma \ref{lem:backwards asymptotic translators},
\[
\s_\ast(w)=H(w)=-\langle w, \vec v\rangle=\s_\ast(w_0)\langle w_0, w\rangle\,.
\]
Combining this with \eqref{eq:vertex of V} and using \eqref{inclusion} we find that for all $w\in G(\Sigma_z)$
\[
\s_\ast(z_V)\langle w, z_V\rangle=\s_\ast(w_0)\langle w_0, w\rangle=\s_\ast(z_V)\langle z_V, w_0\rangle\langle w_0, w\rangle\,.
\]
This implies either that $w_0= z_V$ or that $w_0=z$. In the latter case we have $G(\Sigma_z)\subset \{az+ b e_1:a, b\in\R\}$ and thus $\Sigma_z$ is a Grim hyperplane. Since the speed of the Grim hyperplane is given by $\s_\ast(z)<1$, this is impossible since $\Sigma_z$ lies in a slab of width $\pi$. Therefore \eqref{eq:tip2} holds which implies that $z_0\in G(\Sigma_z)$.
\end{proof}

\begin{claim}\label{k>k_0}
Let $V$ be a $k_0$-face, $k_0\le n-1$. If $G(\Sigma_z)= \relint G(V)$ for all $z\in \relint G(V)$ and all $k$-faces $V$ with $k\in\{0,\dots,k_0-1\}$, then $G(\Sigma_V)=\relint G(V)$ for any $\Sigma_V$ corresponding to $z_V$  and any $z\in\relint G(V)$ is in the Gauss image of some asymptotic translator corresponding to $z_V$.
\end{claim}
\begin{proof}
Choose $z\in \relint G(V)$ and let $\Sigma_z= \lim_{i\to\infty}(\M_{t_i}-X(z^i, t_i))$, for a sequence of normals $z^i\to z$ and a sequence of times $t_i\to-\infty$, be an asymptotic translator corresponding to $z\in\relint G(V)$ (obtained as in Lemma \ref{lem:backwards asymptotic translators}).
First note that
\begin{equation}\label{inclusion2}
G(\Sigma_z)\subset G(V)\,.
\end{equation}
Indeed, if not, then we can find $z_0\in G(\Sigma_z)\setminus G(V)$. Since $z_0\in G(\Sigma_z)$, there exists a sequence of normals $z_i\in G_\ast$ with $z_i\to z_0$ such that $|X(z_i, t_i)- X(z^i, t_i)|$ is uniformly bounded. In particular,
\begin{equation}\label{close2}
(-t_i)^{-1}|X(z_i, t_i)- X(z^i, t_i)|\to 0\,\,\text{ as } i\to \infty\,.
\end{equation}
Let $V_0\in F^\ell_P$, $0\le \ell\le n$, be the $\ell$-face of $P$ for which $z_0$ is in the relative interior of $G(V_0)$. By definition of the squash-down, \eqref{close2} can only happen if $z\in G(V_0)$, and so $\ell\le k-1$. But then \eqref{close2} implies that $z \in G(\Sigma_0)$, where $\Sigma_0$ is a translator corresponding to $z_0$. By the  inductive hypothesis we thus conclude that $z\in \relint G(V_0)$, which is impossible.

As in the end of the proof of  Claim \ref{k=0}, using the fact that the squash-down of $\Sigma_z$ is a circumscribed cone, $\Sigma_z$ is contained in a slab of width $\pi$, the structure of $P$ and \eqref{inclusion2}, we obtain that $G(\Sigma_V)=\relint G(V)$.

Finally, the fact that  $z$ is in the Gauss image of some asymptotic translator corresponding to $z_V$ follows as in  Claim \ref{k=0}. We note here that $\Sigma_z$ can be a Grim hyperplane if $k_0=n-1$, in which case $z= z_V$ so \eqref{eq:tip2} trivially holds.
\end{proof}
The proposition now follows by induction.
\end{proof}

\section{Backwards asymptotics}\label{sec:squash-down}



The squash-downs of all of the examples constructed in \S \ref{sec:existence} are circumscribed. We now show that this is necessary.

\begin{theorem}\label{thm:circumscribed}
Let $\{\M_t\}_{t\in(-\infty,\,\omega)}$ be a convex ancient solution to mean curvature flow in $\R^{n+1}$ with $\sup_{\M_t}\vert \mathrm{II}\vert<\infty$ for each $t$. If $\{\M_t\}_{t\in(-\infty,\,\omega)}$ lies in the slab $(-\frac{\pi}{2},\frac{\pi}{2})\times \R^{n}$ and in no smaller slab, then its squash-down $\Omega_*$ circumscribes the unit sphere $\{0\}\times S^{n-1}$.
\end{theorem}

So let $\{\M_t\}_{t\in(-\infty,\,\omega)}$ be any convex ancient solution to mean curvature flow in $\R^{n+1}$ with $\sup_{\M_t}\vert \mathrm{II}\vert<\infty$ for each $t$ which sweeps-out the slab $(-\frac{\pi}{2},\frac{\pi}{2})\times \R^{n}$.

Given $A\in \Omega_\ast$ and $t\in(-\infty,\omega)$, define the \textsc{width} $w_A(t)$ of $\M_t$ above the point $-t A$ by
\[
w_A(t)\doteqdot \vert \Omega_t\cap (-tA+\R e_1)\vert\,,
\]
where $\Omega_t$ is the convex body bounded by $\M_t$, and denote by 
\begin{align*}
\M_{t}^\pm\doteqdot\big\{p\in\M_{t}:\pm\inner{\nu(p,t)}{e_1}>0\big\}
\end{align*}
the component of $\M_{t}$ on which the normal $\pm\nu$ is ``upward pointing''. 

\begin{lemma}\label{lem:width estimate}
For every $A\in \relint\Omega_\ast$,
\[
\lim_{t\to-\infty} w_A(t)= \pi\,.
\]
\end{lemma}
\begin{proof}
Fix $A\in \relint \Omega_\ast$ and any $\rho\in(0,\rho_A)$, where $\rho_A\doteqdot \mathrm{dist}(A,\partial\Omega_\ast)$.
Given $R>0$ and $h<1$, let $\Pi^{h,R}$ be the rotation about the $e_{1}$-axis of the scaled Angenent oval $h^{-1}\mathrm{A}_{-h^2R}$. For each $t<0$, we may choose $R=R(t)$ so that $\Pi^{h, R}\cap \{x_1=0\}= S^{n-1}_{-\r t}$. Since $h<1$ and $\rho<\rho_A$, the hypersurface $\Pi^{h,R}-tA$ intersects both $\M_{t}^+$ and $\M_{t}^-$ for all $t$. 

By \eqref{eq:Angenent oval kappa}, the mean curvature of $\Pi^{h,R}$ is given by
\begin{align*}
H={}&h\sqrt{\cos^2\theta+a_h^2(-R)}+\cos\theta\frac{n-1}{y(\theta,-R)},
\end{align*}
where $a_h^2(t)\doteqdot (e^{2h^2t}-1)^{-1}$, $\theta$ is the turning angle of the cross-section, and
\[
y(\theta,t)\doteqdot -ht+\frac{1}{h}\log\left(\frac{\sqrt{\cos^2\theta+a_h^2(t)}+\cos\theta}{\sqrt{1+a_h^2(t)}} \right).
\]
Translate $\Pi^{h,R}-tA$ downwards until it is tangent to $\M^+_{t}$. At the point $p=p(t)$ of tangency, the normal $z=z(t)$ to $\M_{t}$ is given by
\begin{equation}\label{eq:zj}
z=\sin\theta e_1+\cos\theta\phi
\end{equation}
and the mean curvature is bounded by
\begin{equation}\label{H-upbound width}
H(p,t)\le h\cos\theta\left(\sqrt{1+\frac{a_h^2(-R)}{\cos^2\theta}}+\frac{n-1}{hy(\theta,-R)}\right)
\end{equation}
for some $\theta=\theta(t)\in(0,\frac{\pi}{2}]$ and some $\phi=\phi(t)\in e_1^\perp$. On the other hand, the differential Harnack inequality and \eqref{H_ast=sigma_ast} imply that
\begin{equation}\label{eq:inscribed contra width}
H(p,t)\ge H_\ast(z)=\sigma_\ast(z)\ge \cos\theta\,\sigma_\ast(\phi)\ge \cos\theta\,.
\end{equation}
Combining \eqref{H-upbound width} with \eqref{eq:inscribed contra width} yields
\[
\frac{a_h(-R)}{\cos\theta}+\frac{n-1}{hy(\theta,-R)}\ge \frac{1}{h}-1>0\,.
\]
Since
\[
R=-\frac{\rho}{h}t+O(1)
\]
as $t\to-\infty$, we conclude, as in \eqref{eq:cos theta small}, that
\begin{subequations}
\begin{align}
\cos\th\le{}& C_he^{h^2t}\\
y(\th,-R)\le{}& C_h\,,\label{eq:squash down y est}
\end{align}
\end{subequations}
where $C_h$ is a constant which depends only on $h$. 

Let $q=q(t)$ be the point of intersection of $\M^+_{t}$ and $\R e_1$. Since $q$ lies below the tangent hyperplane to $\M_{t}$ at $p$, we obtain
\[
\inner{p-q}{z}\ge 0\,,
\]
which, recalling \eqref{eq:zj}, yields
\begin{align*}
\inner{p}{e_1}\ge{}& \inner{q}{e_1}-\inner{p}{\phi}\cot\theta\\
\ge{}&\inner{q}{e_1}-(-t|A|+C_h)\frac{C_he^{h^2t}}{\sqrt{1-C_he^{h^2t}}}\underset{t\to-\infty}{\longrightarrow} \frac{\pi}{2}\,.
\end{align*}
Since $\langle p, e_1\rangle\le  \langle p_A, e_1\rangle $, where $p_A=\M_t\cap(-tA+\R e_1)$, the result follows.
\end{proof}

\begin{lemma}\label{lem:hast regular point}
If the halfspace $\{p\in\R^{n+1}:\inner{p}{z_\ast}\leq h_\ast\}$, where $z_\ast\in \{0\}\times S^{n-1}$, supports $\Omega_\ast$ at a regular point, $p_\ast\in\partial\Omega_\ast$, then $h_\ast=1$.
\end{lemma}
\begin{proof}
By Proposition \ref{prop:sigma_ast ge 1}, it suffices to show that $h_\ast\le 1$. Assume, to the contrary, that $h_\ast>1$. Let $S_{\rho_\ast}(A_\ast)$, where $A_\ast\doteqdot p_\ast-\rho_\ast z_\ast$, be the largest sphere contained in $\Omega_\ast$ which touches the boundary at $p_\ast$ and consider, for any $\rho<\rho_\ast$, the sphere $S_{\rho}(A)$, where $A\doteqdot A_\ast+ \frac{\rho_\ast-\rho}{2}z_\ast$.

\begin{figure}[h]
\centering
\includegraphics[width=0.55\textwidth]{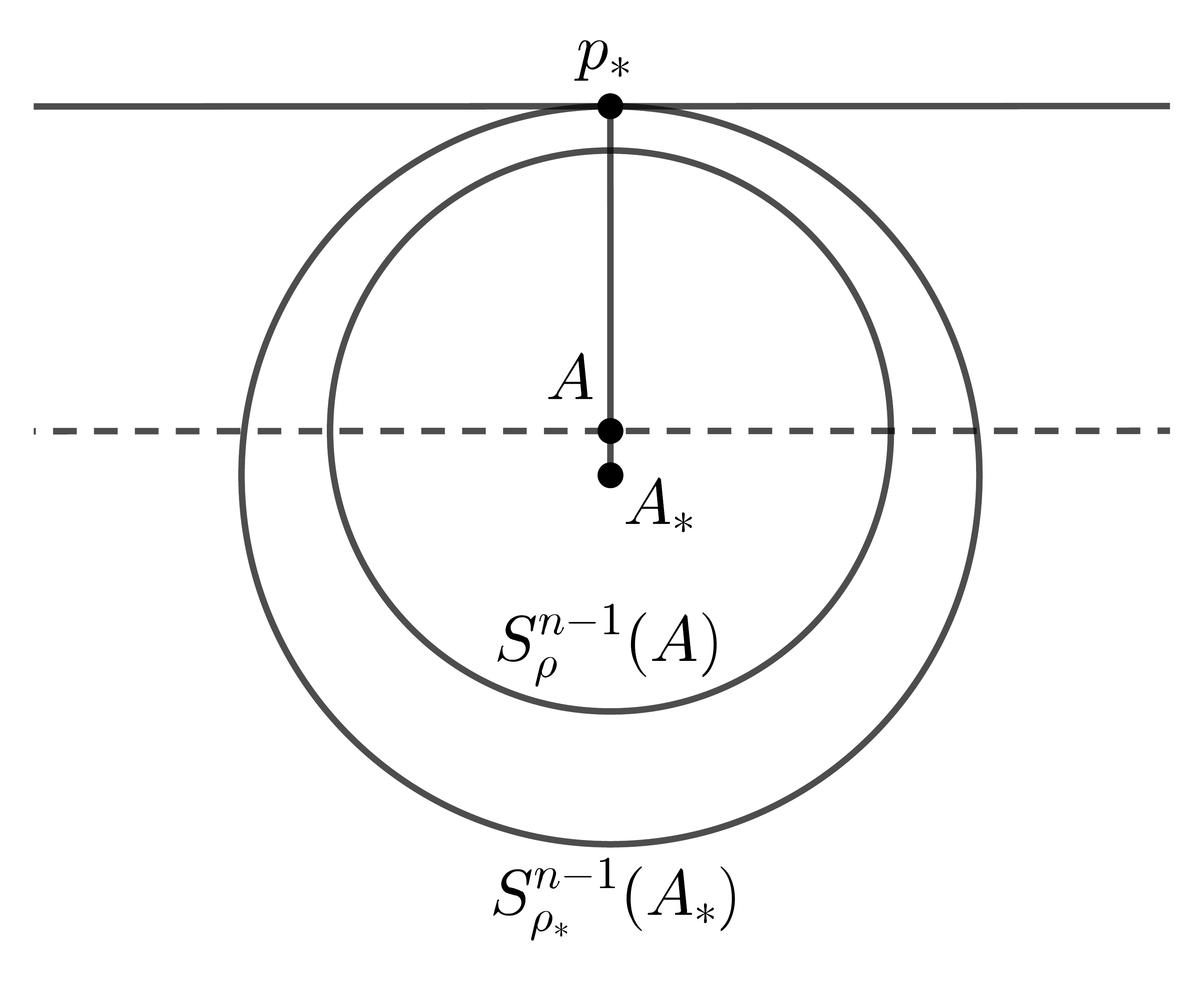}
\caption{\footnotesize If $\sigma_\ast(z_\ast)>1$, then the tangent plane to $\M_t$ with normal $z_\ast$ reaches $-t A$ before the pancake centred at $-t A$ contracts to a point, violating the avoidance principle.}
\end{figure}

Given $h\in (1, h_\ast)$ and $R>0$, denote by $\Pi^{h}_{-R}\doteqdot h^{-1}\Pi_{-h^2R}$ the time $-R$ slice of the scaled rotationally symmetric ancient pancake. Given any $t<0$, we may choose $R_t$ so that $\Pi^{h}_{-R_t}\cap \{x_1=0\}= S^{n-1}_{-\r t}$.
Then, by Lemma \ref{lem:width estimate}, $\Pi^{h}_{-R_t}-tA$ lies to the inside of $\M_{t}$ for $-t$ sufficiently large.

On the other hand, if $t'$ is the time at which
\[
\sigma(z_\ast,t')=\sigma(z_\ast,t)+\rho_\ast t\,,
\]
then
\begin{align*}
-\rho_\ast t
={}&\int_{t}^{t'}H(z_\ast,s)ds\ge\sigma_\ast(z_\ast)(t'-t)=h_\ast(t'-t)\,.
\end{align*}
That is, the tangent hyperplane to $\M_t$ with normal $z_\ast$ moves a distance at least $-\rho_\ast t$ after time
\[
T\doteqdot-\frac{\rho_\ast}{h_\ast}t\,.
\]
Since, by \eqref{eq:ancient pancake radius}, $\Pi^{h}_{R_t}$ exists for time 
\begin{align*}
-\alpha_{h,R_t}={}&-\frac{\rho}{h}t+o(t)\;\;\text{as}\;\; t\to-\infty
\end{align*}
under mean curvature flow, if we choose
\begin{equation}\label{eta}
\rho>\frac{h}{h_\ast}\rho_\ast, 
\end{equation}
then, for $-t$ sufficiently large, the evolution of the initial hypersurface $\Pi^{h}_{R_t}+tA$ has not yet contracted to a point after elapsed time $T$. Since
\[
\sigma(z_\ast,t')=-(h_\ast-\rho_\ast)t+o(t)\;\;\text{as}\;\; t\to-\infty\,,
\]
this violates the avoidance principle.
\end{proof}

Theorem \ref{thm:circumscribed} follows, since every convex body is the intersection of its tangent half-spaces at regular points.

\begin{corollary}\label{cor:translators}
Let $\Sigma$ be a convex translator which lies in a slab region in $\R^{n+1}$.
\begin{enumerate}
\item If $\Sigma$ sweeps out the slab $(-\frac{\pi}{2},\frac{\pi}{2})\times \R^n$, then
\[
\overline{G(\Sigma)\cap(\{0\}\times \R^n)}\subset \big\{z\in \{0\}\times S^{n-1}:\inner{z}{\vec v}<0\big\}\,.
\]
\item If the squash-down of $\Sigma$ splits off a line, then $\Sigma$ splits off the same line. \label{item:translator splitting}
\item $\inf_{z\in G(\Sigma)}\inner{z}{\vec v}$ is attained.\label{item:tip attained}
\end{enumerate}
\end{corollary}
\begin{proof}
Recalling Proposition \ref{prop:translator squash-down}, the claims follow readily from Theorem \ref{thm:circumscribed}.
\end{proof}

Next, we study the asymptotic translators. We need to control the Gauss image in the noncompact case.

\begin{proposition}\label{constant Gauss}
Let $\{\M_{t}\}_{t\in(-\infty,\,\omega)}$ be a maximal convex ancient solution to mean curvature flow which lies in $(-\frac{\pi}{2},\frac{\pi}{2})\times \R^{n}$ and in no smaller slab and satisfies $\sup_{\M_t}\vert \mathrm{II}\vert<\infty$ for each $t$. If its squash-down $\Omega_\ast$ is unbounded but nondegenerate, then $\{\M_{t}\}_{t\in(-\infty,\,\omega)}$ is eternal and its Gauss image is constant. If $\Omega_\ast$ is degenerate, then there exists $t_0>-\infty$ such that $G(\M_t)$ is constant for $t<t_0$.
\end{proposition}
\begin{proof}
We may assume, without loss of generality, that $\M_t$ is locally uniformly convex for each $t$.

We first show that the Gauss image is constant. Since $G(\M_t)$ is nondecreasing, it suffices to show that $G(\M_t)\subset G(\Omega_\ast)$ for each $t$ (sufficiently early in the degenerate case). Choose a sequence of times $t_j\to-\infty$ and consider, for each $j\in\N$ and $r>0$, the set
\[
P^j_r\doteqdot \bigcap_{z\in F_j}\big\{p\in \{0\}\times \R^n:\inner{p}{z}<-t_j-r\big\}\,,
\]
where $F_j$ is the set of outward unit normals to supporting hyperplanes for $\Omega_{t_j}$. Note that 
$G(P^j_r)=G(\M_{t_j})$ for all $r>0$. By Lemma \ref{lem:width estimate}, for every $h>1$ we can find $j\in\N$ such that, for each $r>-(h-1)t_j$ and $p\in P_r^j$, the translated time $t_{r,h}$ slice of the rescaled rotationally symmetric ancient pancake, $h^{-1}\Pi_{h^2t_{r,h}}-p$, lies to the inside of $\M_{t_j}$, where $t_{r,h}$ is the time at which $h^{-1}\Pi_{h^2t_{r,h}}$ has horizontal displacement $r+(h-1)t_j$. The avoidance principle then implies that
\[
P_r^j\subset \Omega_{t_j-t_{r,h}}\,.
\]
In the non-degenerate case, the claim follows since for each $t>t_j$ we can find $r$ sufficiently large that $t_j-t_{r,h}\ge t$, and hence
\[
G(\M_t)\subset G(P^j_r)=G(\M_{t_j})\,.
\]

If $\Omega_\ast$ is degenerate, then we can certainly not take $r$ as large as we want, since for $r>-t_j$ we have $P^j_r=\emptyset$. In this case, however, we can consider $t_j=-j$, $h\in (1,2)$ and applying the above argument for any $t\in(0,1]$ we obtain that, for $j$ large enough $G(\M_{t_{j-1}})\subset G(\M_{t_j})$, which along with the monotonicity of $G(\M_t)$ yields that the Gauss image of $\M_t$ is constant for all $t$ early enough.

We now show that $\omega=\infty$ when $\Omega_\ast$ is non-degenerate. Given any time $T<\omega$ and any ``height'' $L>0$, we may ``double'' $\M_{T}$ at height $L$ as in Lemma \ref{lem:old-but-not-ancient} to obtain a compact solution $\{\M^{T,L}_t\}_{t\in[T,\,\omega_{T,L})}$. Since $\Omega_\ast$ is nondegenerate and the Gauss image of $\M_t$ is constant, we deduce (using arbitrarily large rotationally symmetric pancakes as inner barriers as in Lemma \ref{lem:old-but-not-ancient}) that $\omega_{T,L}\to\infty$ as $L\to\infty$. Taking a limit along an appropriate sequence of heights $L_j\to\infty$, we obtain an eternal solution $\{\M^{T}_t\}_{t\in(-\infty,\infty)}$ such that $\M_t^T=\M_t$ for each $t\le T$. Taking a limit along an appropriate sequence of times $T_j\to \omega$ yields an eternal solution $\{\M^\infty_t\}_{t\in(-\infty,\infty)}$ such that $\M^\infty_t=\M_t$ for each $t\le \omega$. Since $\omega$ is maximal, we conclude that $\omega=\infty$.
\end{proof}

We can now obtain the desired structure result for the backward asymptotic translators.

\begin{proposition}\label{prop:converging sequences}
Let $\{\M_t\}_{t\in(-\infty,\,\omega)}$ be a convex ancient solution to mean curvature flow with $\sup_{\M_t}\vert \mathrm{II}\vert<\infty$ for each $t$ which lies in $(-\frac{\pi}{2},\frac{\pi}{2})\times \R^{n}$ and in no smaller slab. Consider a region $W\subset \R^{n+1}$ such that $\Omega_\ast\cap W$ is bounded, where  $\Omega_\ast\subset \{0\}\times\R^n$ is the squash-down of $\{\M_t\}_{t\in(-\infty,\,\omega)}$. Given a sequence of times $t_i\to -\infty$ and a sequence of points $p_i\in \M_{t_i}\cap  W$,  consider the sequence of flows $\{\M^i_s\}_{s\in (-\infty,\,\omega-t_i)}$ defined by $\M^i_s\doteqdot \M_{t_i+s}-p_i$. After passing to a subsequence, these flows converge locally uniformly in the smooth topology to a translating solution which satisfies
\[
H(z, t)=\s_\ast (z)\,,
\]
where $z=\lim_{i\to\infty}\nu(p_i,t_i)$.
\end{proposition}
\begin{remark} We remark that $W$ is allowed to have any Hausdorff dimension and also it could be the whole $\R^n$ in case $\Omega_\ast$ is compact.
\end{remark}

\begin{proof}[Proof of Proposition \ref{prop:converging sequences}]
By the definition of the squash-down, for any $\d>0$, there exists $t_\d$ such that for all $t<t_\d$
\begin{equation}\label{genbound}
\M_t\cap W\subset (-\pi/2, \pi/2)\times ((1+\d)(-t) \Omega_\ast\cap W)\,.
\end{equation}
Moreover, by Proposition \ref{constant Gauss}, $G(\M_t)= G_\ast$ for all $t<t_\d$.
Take a sequence $t_i\to-\infty$ and consider points
\[
p_i\in \M_{t_i}\cap W\,.
\]
For these points we have
\[
p_i=-t_i a_i\ell_i+ b_ie_1\,,\,\,\text{ where }\ell_i\in \partial \Omega_\ast\,, a_i\in[0, 1+\d_i)\,, b_i\in(-\pi/2,\pi/2)\,,
\]
with $\d_i\to0$ by \eqref{genbound} 
and denote their normals by
\[
z_i=\cos\phi_iw_i+\sin\phi_i e_1\,,
\]
 where  $ \phi_i\in[-\pi/2,\pi/2]$ and $w_i\in G(\M_{t_i})\cap\{z:z\cdot e_1=0\}$ when $\phi_i\ne \pm\pi/2$. After passing to a subsequence we can arrange that 
$z_i$ converges to some normal $z$ with
\[
z=\cos\phi_\infty w_\infty+\sin \phi_\infty e_1\,,
\]
where 
\[ 
\phi_{i}\to \phi_\infty \in[-\pi/2,\pi/2]\,.
\]
In case $\phi_\infty\ne \pm\frac\pi2$, $w_\infty$ is a unit normal to $\partial \Omega_\ast$ at a point in $W$ and
\[ 
w_i\to w_\infty\,.
\]

For the support function, we have
\begin{equation}\label{sigmaest}
\begin{split}
\s(z_i,t_i)=\langle p_i, z_i\rangle={}& -t_i a_i\langle \ell_i, w_i\rangle \cos\phi_i+ b_i\sin\phi_i\\
\le{}& -t_ia_i\s_\ast(w_i)\cos\phi_i+ b_i\sin\phi_i\,.
\end{split}
\end{equation}
Note also that, since the Gauss image of $\M_t$ is constant for $t$ small enough, $\s_\ast(w_i)<\infty$.
For $t_j>t_i$, the evolution equation \eqref{eq:MCF support} and the Harnack inequality \eqref{eq:Harnack} yield
\[
\s(z_i,t_j)-\s(z_i,t_i)
<-H(z_i, t_i)(t_j-t_i)
\]
and hence
\[
\lim_{i\to\infty}H(z_i, t_i)\le \lim_{i\to\infty}\frac{\s(z_i,t_i)}{-t_i}\,.
\]
Using \eqref{sigmaest}, this implies that if $\cos\phi_\infty=0$, then $\lim_i H(z_i, t_i)=0$ and if $\cos\phi_\infty\ne 0$ then 
\[
\begin{split}
\lim_{i\to\infty} H(z_i, t_i)&\le \lim_{i\to\infty} a_i\s_\ast(w_i)\cos\phi_i
\le \cos\phi_\infty\s_\ast(w_\infty)\,.
\end{split}
\]
On the other hand, for the latter case and for all $i$, by the Harnack inequality, we know that $H(z_i, t_i)\ge\s_\ast (z_i)=\cos\phi_i\s_\ast(w_i)$ and taking limits
\[
\lim_{i\to\infty}H(z_i, t_i)\ge \cos\phi_\infty\s_\ast(w_\infty)\,.
\]
Thus,
\begin{equation}\label{TheHlimit}
\lim_{i\to\infty}H(z_i, t_i)= \begin{cases}
\cos\phi_\infty\s_\ast(w_\infty)&\text{ if }\cos\phi_\infty\ne 0\\
0&\text{ if }\cos\phi_\infty= 0.
\end{cases}
\end{equation}

After passing to a subsequence, we obtain a limit flow $\{\M^\infty_s\}_{s\in(-\infty, \infty)}$ defined by
\[
\M^\infty_s=\lim_{i\to\infty}(\M_{t_i+s}-p_i)\,.
\]
Since $p_i\in \M_{t_i}$ and $z_i\to z$ we find that $z\in G(\M^\infty_0)$ and  so $z \in G(\M^\infty_s)$ for all $|s|$ small enough. There exist thus $p_i^s\in \M_{s+t_i}$ with normals $z_i^s\to z$ such that $X(z_i^s, t_i+s)-p_i$ converges, as $i\to\infty$, to a point in $\M^\infty_s$ whose normal is $z$. 
Arguing as above, replacing $t_i$ with $t_i+s$ and $z_i$ with $z_i^s$ (and since $\sup_i|p_i^s-p_i|<\infty$) we find that \eqref{TheHlimit} holds with these replacements and hence, for all $|s|$ small,
\begin{equation}\label{TheHestimate}
\lim_{i\to\infty}H(z_i^s, t_i+s)_= \begin{cases}
\cos\phi_\infty\s_\ast(w_\infty)&\text{ if }\cos\phi_\infty\ne 0\\
0&\text{ if }\cos\phi_\infty= 0.
\end{cases}
\end{equation} 
We conclude that $\{\M^\infty_s\}_{s\in(-\infty, \infty)}$ is a translating solution that satisfies 
\[
H(z,s)=\s_\ast(z)\,.\qedhere
\]
\end{proof}

\section{Forwards asymptotics}\label{sec:forward asymptotics}


Recall that the ``exscribed'' body $P_{-1}$ corresponding to a circumscribed convex body $P$ 
in $\R^n$ is defined by
\[
P_{-1}\doteqdot \bigcap_{z\in F}\{p\in \R^n:\inner{p}{z}\le -1\}\,,
\]
where $F$ is the set of unit outward normals to supporting halfspaces of $P$ which support $S^{n-1}$. We will show that the forward squash-down of a general eternal example is the exscribed body of its (backward) squash-down. 

\begin{figure}[h]
\centering
\includegraphics[width=0.75\textwidth]{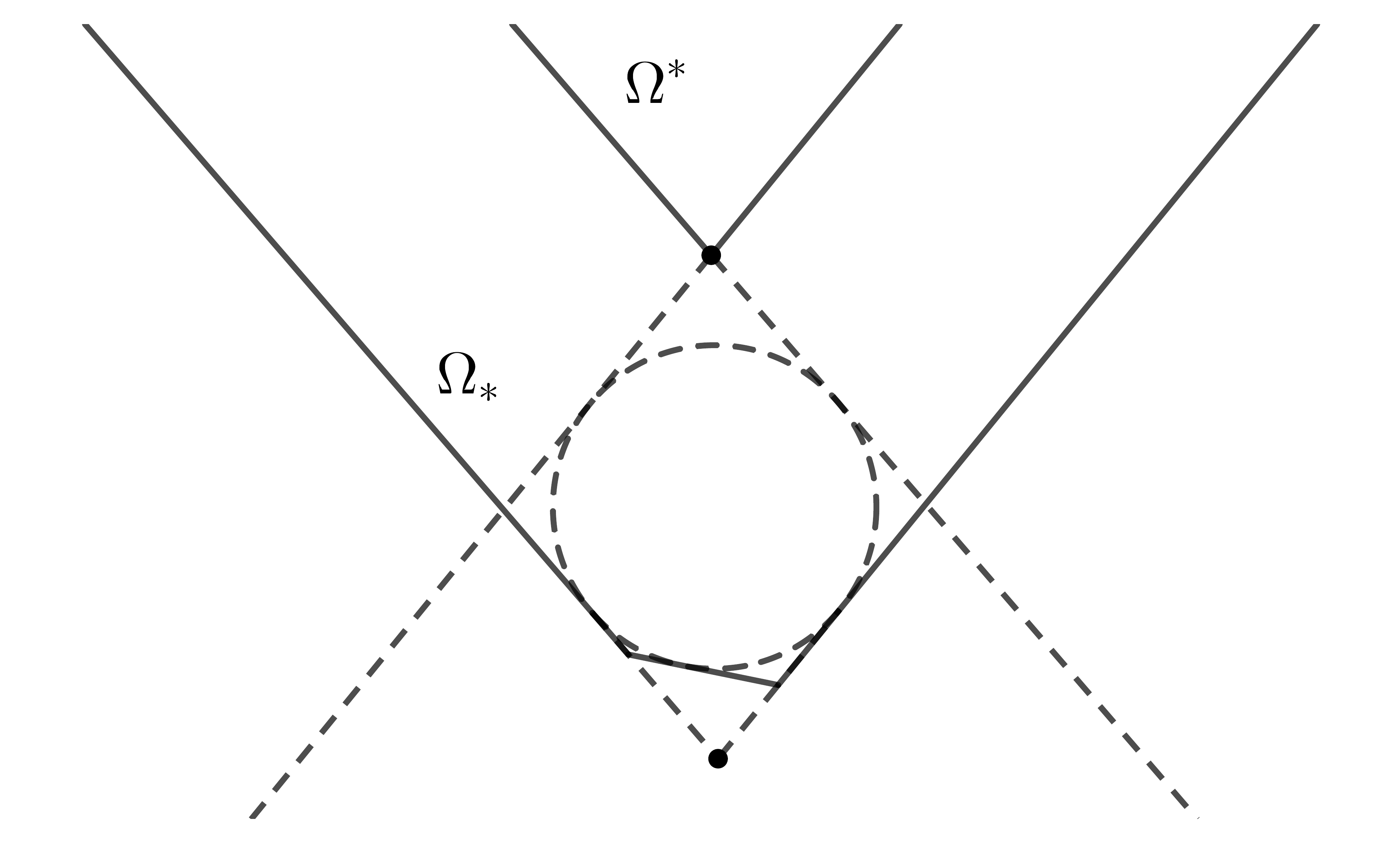}
\caption{The backward and forward squash-downs $\Omega_\ast$ and $\Omega^\ast$, respectively, of an eternal solution $\{\M\}_{t\in(-\infty,\infty)}$.}\label{fig:Gauss image}
\end{figure}

\begin{theorem}\label{thm:squash-up}
Let $\{\M_{t}\}_{t\in(-\infty,\infty)}$, $\M_t=\partial\Omega_t$, be a convex eternal solution with $\sup_{\M_t}\vert \mathrm{II}\vert<\infty$ for each $t$. If $\{\M_t\}_{t\in(-\infty,\,\infty)}$ lies in a stationary slab region, then
\[
\Omega^\ast=(\Omega_\ast)_{-1}\,.
\]
\end{theorem}
\begin{proof} 
Since $\sigma$ is concave in $t$,
\[
\frac{\sigma(z,t)-\sigma(z,s)}{t-s}\le \partial_t\sigma(z,t)=-H(z,t)\le -H_\ast(z)=-\sigma_\ast(z)
\]
for all $z$ and all $t<s$. Theorem \ref{thm:circumscribed} then implies that
\begin{equation}\label{eq:squash-up general case 1}
\sigma^\ast(z)\le -1
\end{equation}
for each $z$. In particular, if $\Omega_\ast$ is degenerate, then we obtain
\[
\Omega^\ast=\emptyset=(\Omega_\ast)_{-1}\,.
\]
So we may assume that $\Omega_\ast$ is nondegenerate.

Given $R>0$ and $r>0$, set
\[
P^R_r\doteqdot \bigcap_{z\in F_\ast}\{p\in \{0\}\times \R^n:\inner{p}{z}<R-r\}\,,
\]
where $F_\ast$ is the set of outward unit normals to supporting hyperplanes for $\Omega_\ast$. By Lemma \ref{lem:width estimate}, for every $h>1$ we can find $R>0$ such that, for each $r>(h-1)R$ and $p\in P_r^R$, the translated time $t_{r,h}$ slice of the rescaled rotationally symmetric ancient pancake, $h^{-1}\Pi_{h^2t_{r,h}}-p$, lies to the inside of $\M_{-R}$, where $t_{r,h}$ is the time at which $h^{-1}\Pi_{h^2t_{r,h}}$ has horizontal displacement $r-(h-1)R$. The avoidance principle then implies that
\[
P_r^R\subset \Omega_{-R-t_{r,h}}\,.
\]
Since $t_{r,h}=-\frac{r}{h}+o(r)$ as $r\to\infty$, we find that
\[
P_{-1}=\lim_{r\to\infty}\frac{1}{r}P_r^R\subset \lim_{r\to\infty}\frac{1}{r}\Omega_{-R-t_{r,h}}=\frac{1}{h}\Omega^\ast\,.
\]
Since $h>1$ was arbitrary, we conclude that
\begin{equation}\label{eq:squash-up general case 2}
P_{-1}\subset\Omega^\ast\,.
\end{equation}
Together, \eqref{eq:squash-up general case 1} and \eqref{eq:squash-up general case 2} imply the claim.
\end{proof}

\begin{corollary}\label{cor:wedge implies translation}
Let $\{\M_t\}_{t\in(-\infty,\,\omega)}$ be a convex ancient solution to mean curvature flow with $\sup_{\M_t}\vert \mathrm{II}\vert<\infty$ for each $t$ which lies in a slab. If the squash-down of $\{\M_t\}_{t\in(-\infty,\,\omega)}$ is a cone, then $\{\M_t\}_{t\in(-\infty,\,\omega)}$ is a translator.
\end{corollary}
\begin{proof}
See the proof of Theorem \ref{thm:unbounded polygonal pancakes}.
\end{proof}

Recalling Corollary \ref{cor:translators} \eqref{item:translator splitting} we obtain a useful uniqueness statement for the Grim hyperplane.
\begin{corollary}[Grim uniqueness]\label{cor:Grim uniqueness}
Let $\{\M_t\}_{t\in(-\infty,\,\omega)}$ be a convex ancient solution to mean curvature flow with $\sup_{\M_t}\vert \mathrm{II}\vert<\infty$ for each $t$ which lies in a slab. If the squash-down of $\{\M_t\}_{t\in(-\infty,\,\omega)}$ is a half-space, then $\{\M_t\}_{t\in(-\infty,\infty)}$ is a Grim hyperplane.
\end{corollary}

\begin{remark}\label{rem:wing uniqueness}
A similar argument shows that the flying wing translators constructed in \cite{BLT2} are uniquely determined (amongst rotationally symmetric examples when $n\ge 3$) by their squash-downs, since they are uniquely determined by their asymptotic translators \cite{HIMW}. It remains an open question in general whether or not convex translators in slabs are uniquely determined by their asymptotic translators (the argument of \cite{HIMW} is inherently two-dimensional).
\end{remark}

We now obtain a useful structure result for the forward asymptotic translators (cf. Proposition \ref{prop:converging sequences}).

\begin{proposition}\label{prop:converging sequences nc}
Let $\{\M_t\}_{t\in(-\infty,\infty)}$ be a convex eternal solution to mean curvature flow in $\R^{n+1}$ with $\sup_{\M_t}\vert \mathrm{II}\vert<\infty$ for each $t$ and such that it sweeps-out the slab $(-\frac{\pi}{2},\frac{\pi}{2})\times \R^{n}$. Given a sequence $t_i\to +\infty$ and $p_i\in \M_{t_i}$ consider the sequence of flows $\M^i_s=\M_{t_i+s}-p_i$. Then, after passing to a subsequence, these flows converge to a translating solution which satisfies 
\[
H(z, t)=\s^\ast (z)\,,
\]
where $z=\lim z_i$ and $z_i$ is the outward unit normal to $\M_{t_i}$ at $p_i$.
\end{proposition}
\begin{proof}
Recall that for the squash-down and squash-forward, by Theorems \ref{thm:squash-up} and \ref{constant Gauss}, we have that $\Omega^\ast=(\Omega_\ast)_{-1}$ and $G(\Omega_\ast)=G(\Omega^\ast)= \ov {G(\M_t)}$ fof all $t\in (-\infty, \infty)$. The proof proceeds similarly to that of Proposition \ref{prop:converging sequences}.

Consider a sequence of times $t_i\to +\infty$, points $p_i\in \M_{t_i}$
 and denote their normals by 
\[
z_i=\cos\phi_iw_i+\sin\phi_i e_1\,.
\]
where $ \phi_i\in(-\pi/2,\pi/2)$ and  $w_i\in G(\M_{t_i})\cap\{z:z\cdot e_1=0\}$.
After passing to a subsequence we can assume that 
\[
z_i\to z=\cos\phi_\infty w_\infty+\sin\phi_\infty e_1\,,
\]
where $\phi_i\to\phi_\infty\in[-\pi/2,\pi/2]$, and $w_i\to w_\infty\in G(\Omega^\ast)$ when $\phi_\infty\ne \pm \pi/2$. Recall that by \eqref{eq:forward backward limits}, for all $z\in G^\ast$, $\frac{\s(z,t)}{-t}$ is increasing and 
\[
\lim_{t\to \infty}H(z,t)=H^\ast(z)=-\s^\ast(z)=\lim_{t\to\infty}\frac{-\s(z,t)}{t}\,.
\]
Along with the Harnack inequality we thus have
\[
\lim_{i\to\infty}H(z_i, t_i)\le \lim_{i\to\infty} H^\ast(z_i)=-\lim_{i\to\infty}\s^\ast(z_i)=-\s^\ast(z)\,.
\]
This implies that if $\cos\phi_\infty=0$, then $\lim_{i\to\infty} H(z_i, t_i)=0$, so we consider the case $\cos\phi_\infty\ne 0$.
For $t_i>t_j$, the evolution equation \eqref{eq:MCF support} and the Harnack inequality \eqref{eq:Harnack} yield
\[
\s(z_i,t_i)-\s(z_i,t_j)
\ge-H(z_i, t_i)(t_i-t_j)
\]
and hence
\[
H(z_i, t_i)(t_i-t_j)\ge \s(z_i,t_j)-\s(z_i,t_i)
\]
which yields
\[
\lim_{i\to\infty} H(z_i, t_i)\ge \lim_{i\to\infty}\frac{-\s(z_i, t_i)}{t_i}
\]
Now if $z\in G^\ast$, then  $w_\infty\in G(\partial\Omega^\ast)\setminus\partial G(\partial\Omega^\ast)$ and, by convexity, we have that $\frac{p_i}{t_i}$ converges to a point on $\partial \Omega^\ast$ whose normal is $w_\infty$. Indeed, let $q_i\in \M_{t_i}$ be such that $\frac{q_i}{t_i}$ converges to a point on $\partial \Omega^\ast$ whose normal is $w_\infty$. By convexity, $\langle q_i, -z_i\rangle\ge \langle p_i, -z_i\rangle$  and thus $\langle \frac{p_i}{t_i}, -z_i\rangle$ is uniformly bounded. By definition of the squash-down we obtain that $\frac{p_i}{t_i}$ converges. That yields
\[
\lim_{i\to\infty}\frac{\s(z_i, t_i)}{t_i}= \lim_{i\to\infty}\cos\phi_i\left\langle \frac{p_i}{t_i}, w_i\right\rangle=\cos\phi_\infty\s^\ast(w_\infty)
\]
If $z\notin G^\ast$, then $w_\infty\in \partial G(\partial\Omega^\ast)$ and 
\[
 \lim_{i\to\infty} H(z_i, t_i)\ge \lim_{i\to\infty}\s_\ast(z_i)=\s_\ast(z)=-\s^\ast(z)=-\cos\phi_\infty\s^\ast(w_\infty)\,.
\]
We have proven that
\begin{equation}\label{TheHlimitnc}
\lim_{i\to\infty}H(z_i, t_i)= \begin{cases}
\cos\phi_\infty\s^\ast(w_\infty)&\text{ if }\cos\phi_\infty\ne 0\\
0&\text{ if }\cos\phi_\infty= 0.
\end{cases}
\end{equation}
After passing to a subsequence, we obtain a limit flow $\{\M^\infty_s\}_{s\in(-\infty,\infty)}$ defined by
\[
\M^\infty_s=\lim_{i\to\infty}(\M_{t_i+s}-p_i)\,.
\]
Since $p_i\in \M_{t_i}$ and $z_i\to z$ we find that $z\in G(\M^\infty_0)$ and so $z \in G(\M^\infty_s)$ for all $|s|$ sufficiently small. There exist thus $p_i^s\in \M_{s+t_i}$ with normals $z_i^s\to z$ such that $X(z_i^s, t_i+s)-p_i$ converges, as $i\to\infty$, to a point in $\M^\infty_s$ whose normal is $z$. 
Arguing as above, replacing $t_i$ with $t_i+s$ and $z_i$ with $z_i^s$ (and since $\sup_i|p_i^s-p_i|<\infty$) we find that \eqref{TheHlimitnc} holds with these replacements and hence
\begin{equation}\label{TheHestimatenc} 
\lim_{i\to\infty}H(z_i^s, t_i+s)= \begin{cases}
\cos\phi_\infty\s^\ast(w_\infty)&\text{ if }\cos\phi_\infty\ne 0\\
0&\text{ if }\cos\phi_\infty= 0.
\end{cases}
\end{equation} 
Therefore $\{\M^\infty_s\}_{s\in(-\infty, \infty)}$ is a translating solution that satisfies
\[
H(z,s)=\s^\ast(z)\,.\qedhere
\]
\end{proof}

\section{Nontranslating eternal solutions}\label{sec:eternal}

Given a circumscribed cone $C$ with vertex $v$, consider the circumscribed polytope $P$ obtained by truncating the vertex:
\[
P\doteqdot C\cap\{p\in \{0\}\times\R^n:\inner{p}{\hat v}\le 1\},
\]
where $\hat v\doteqdot v/\vert v\vert$. We will construct an eternal solution to mean curvature flow with squash-down $P$. 
By Theorem \ref{thm:squash-up}, its forward squash-down is $P_{-1}=C$. Thus, by Propositions \ref{prop:converging sequences} and \ref{prop:converging sequences nc}, these solutions look like a family of translators which emerge at $t=-\infty$ and coalesce into a single translator at $t=+\infty$, preserving their total exterior dihedral angles.

\begin{theorem}\label{thm:polygonal translators}
For each regular circumscribed cone $C\in \mathrm{P}{}^n_\ast$ there exists an eternal solution $\{\M_t\}_{t\in(-\infty,\infty)}$ to mean curvature flow which sweeps out $(-\frac{\pi}{2},\frac{\pi}{2})\times \R^n$, is reflection symmetric across the hyperplane $\{0\}\times\R^n$, and whose squash-down is
\[
P\doteqdot C\cap\{p\in \{0\}\times\R^n:\inner{p}{\hat v_C}\le 1\},
\]
where $\hat v_C\doteqdot v_C/\vert v_C\vert$, $v_C$ denoting the vertex of $C$.
\end{theorem}
Before proving Theorem \ref{thm:polygonal translators}, we record an important corollary.
\begin{corollary}
\cite[Conjecture 1]{Wh03} is false: there exist convex eternal solutions to mean curvature flow that do not evolve by translation.
\end{corollary}
\begin{proof}
By Proposition \ref{prop:translator squash-down}, the solutions constructed in Theorem \ref{thm:polygonal translators} do not evolve by translation (except when $\vert v\vert=1$).
\end{proof}

Given $R>0$, let $\{\M^R_t\}_{t\in[0,\infty)}$ be the old-but-not-ancient solution corresponding to $P$ constructed in Lemma \ref{lem:old-but-not-ancient}. By Lemma \ref{lem:squash-up}, the forward squash-down of $\{\M^R_t\}_{t\in[0,\infty)}$ is $P_{-1}$. So, by \eqref{eq:forward backward limits}, the mean curvature $H(\hat v,t)$ increases from 1 at time $t=0$ to $\vert v\vert$ as $t\to\infty$. Given $\varepsilon\in(0,\vert v\vert-1)$, define
\[
t_{\varepsilon,R}\doteqdot \inf\big\{t\in[0,\infty):H(\hat v,t)\ge 1+\varepsilon\;\;\text{for all}\;\; t\in [t_{\varepsilon,R},\infty)\big\}.
\]

\begin{lemma}\label{lem:limit is eternal}
For every $\varepsilon\in(0,\vert v\vert-1)$,
\[
\limsup_{R\to\infty}t_{\varepsilon,R}=\infty.
\]
\end{lemma}
\begin{proof}
Suppose, to the contrary, that
\[
T_\varepsilon\doteqdot \limsup_{R\to\infty}t_{\varepsilon,R}<\infty\,.
\]
On the one hand, since, by \eqref{eq:lower speed bound}, $H(\hat v,s)\ge 1$ for $s\in(0,t_{\varepsilon,R})$, integrating \eqref{eq:MCF support} between $0$ and $t\in(t_{\varepsilon,R},\infty)$ yields
\begin{align}
\sigma_R(\hat v,t)={}&\sigma_R(\hat v,0)-\int_0^{t}H(z,s)\,ds\nonumber\\
={}&R-\int_0^{t_{\varepsilon,R}}H(z,s)\,ds-\int_{t_{\varepsilon,R}}^{t}H(z,s)\,ds\nonumber\\
\le{}&R-t_{\varepsilon,R}-(1+\varepsilon)(t-t_{\varepsilon,R})\nonumber\\
={}&R+\varepsilon t_{\varepsilon,R}-(1+\varepsilon)t\,.\label{eq:eternal t est 1}
\end{align}
On the other hand, by \eqref{eq:Angenent oval x} and \eqref{eq:Angenent oval y} (see \cite[Lemma 2.2]{BLT1}), the rotated time $\log 2-R$ slice $\mathrm{A}_{\log 2-R}$ of the Angenent oval lies to the inside of $\M^R_0$, and thus provides an upper bound for $\sigma_R(\hat v,t)$. Indeed, denote by $\ell_R(t)$ the horizontal displacement of $\{\Pi^R_t\}_{t\in[0,T_R)}$, the evolution by mean curvature of the rotation of $\mathrm{A}_{\log 2-R}$. By estimates obtained in \cite[pp. 31-32]{BLT1}, there exists $C<\infty$ such that
\[
\ell_R(t)\ge T_R-t+(n-1)\log(T_R-t)-C
\]
for all $t<T_R-1$, and
\[
T_R\ge R-(n-1)\log T_R-C
\]
for $R$ sufficiently large, and hence
\begin{align*}
\sigma_R(\hat v,t)\ge{}&\ell_R(t)\\
\ge {}&T_R-t+(n-1)\log(T_R-t)-C\\
\ge {}&R-t+(n-1)\log\left(\frac{T_R-t}{T_R}\right)-2C
\end{align*}
for all $t\in(t_{\varepsilon,R},T_R-1)$ for $R$ sufficiently large. Recalling \eqref{eq:eternal t est 1}, we conclude that
\[
\varepsilon t_{\varepsilon,R}\ge\varepsilon t+(n-1)\log\left(\frac{T_R-t}{T_R}\right)-2C
\]
for all $t\in(T_\varepsilon,T_R-1)$ for $R$ sufficiently large. Setting $t\doteqdot T_{\varepsilon}+\frac{2C}{\varepsilon}+10$ and taking $R\to\infty$, we arrive at a contradiction.
\end{proof}

After space-time translating so that $X_R(\hat v_C,0)=0$ and $H_R(\hat v_C,0)=1+\varepsilon$, Lemma \ref{lem:limit is eternal} implies that an eternal solution is obtained in the limit $R\to\infty$. Since $H(\hat v_C,0)=1+\varepsilon>0$ on the limit, the width estimate (Lemma \ref{lem:width_est}) ensures that it cannot be a Grim hyperplane. The nice structure of $P$ will then allow us to show that the squash-down of the limit is $P$.

\begin{proof}[Proof of Theorem \ref{thm:polygonal translators}]
Fix $\varepsilon\in(0,\vert v_C\vert-1)$. Given $R>0$, consider the old-but-not-ancient solution $\{\M^R_t\}_{(-t_{\varepsilon,R},\infty)}$ obtained by translating the old-but-not-ancient solution of Lemma \ref{lem:old-but-not-ancient} corresponding to $P$ in space and time so that $-t_{\varepsilon,R}$ is the initial time and $X(\hat v_C,0)=0$. Since, by Lemma \ref{lem:limit is eternal}, $\limsup_{R\to\infty}t_{\varepsilon,R}=\infty$, the argument of Theorems \ref{thm:bounded polygonal pancakes} and \ref{thm:unbounded polygonal pancakes} yields a sequence of scales $R_j\to\infty$ such that the sequence of flows $\{\M^{R_j}_t\}_{t\in(t_{\varepsilon,R_j},\,\infty)}$ converges locally uniformly in the smooth topology to a family $\{\M_t\}_{t\in(-\infty,\,\infty)}$ of smooth, convex hypersurfaces $\M_t$ which sweep-out the slab $(-\frac{\pi}{2},\frac{\pi}{2})\times\R^n$ and evolve by mean curvature flow. 

By Theorem \ref{thm:circumscribed}, the squash-down $\Omega_\ast$ of $\{\M_t\}_{t\in(-\infty,\,\infty)}$ circumscribes $\{0\}\times S^{n-1}$. Since, by construction,
\begin{equation}\label{eq:speed>1 eternal}
H(\hat v_C,0)=1+\varepsilon>1\,,
\end{equation}
it follows that $\{\M_t\}_{t\in(-\infty,\,\infty)}$ is not a Grim hyperplane. In particular, by Corollary \ref{cor:Grim uniqueness}, there exists $z_0\in G_\ast\setminus\{\hat v_C\}$ for which $\sigma_\ast(z_0)=1$. 

By \eqref{eq:sigma ratio bound}, given any vertex $v$ of $P$, each of the old-but-not ancient solutions $\{\M^R_t\}_{t\in[-t_{\varepsilon,R},\,\infty)}$ satisfies
\[
\inner{z_0}{v}\le \frac{\sigma_R(z_0,t)}{-t}
\]
for all $R$ and all $t\in(-t_{\varepsilon,R},0)$. Thus,
\[
\inner{z_0}{v}\le1
\]
for all vertices $v\in F^0_P$ of $P$. It follows that $z_0=z_f$ is normal to some facet $f\in F_P^{n-1}$. Since $z_0\neq \hat v_C$, $f$ must be a truncated facet of $C$. 

Since $\{\M_t\}_{t\in(-\infty,\infty)}$ inherits the symmetries of $P$, we conclude that every facet of $C$ lies in a supporting halfspace for $\Omega_\ast$. Since $\sigma_\ast(\hat v)<\vert v\vert$, $\Omega_\ast$ circumscribes $\{0\}\times S^{n-1}$ and every facet of $\Omega_\ast$ with normal not equal to $\hat v$ comes from a facet of $C$, we conclude that $\Omega_\ast=P$. 
\end{proof}

\section{Reflection symmetry}\label{sec:reflection symmetry}

Finally, we exploit Theorems \ref{thm:circumscribed} and \ref{thm:squash-up} and Propositions \ref{prop:converging sequences} and \ref{prop:converging sequences nc} to prove that convex ancient solutions in slabs in $\R^3$ are symmetric under reflection across the mid-hyperplane.

\begin{theorem}\label{thm:reflection symmetry}
Let $\{\M_t\}_{t\in(-\infty,\,\omega)}$ be a convex ancient solution to mean curvature flow in $\R^{3}$ with $\sup_{\M_t}\vert \mathrm{II}\vert<\infty$ for each $t$ which lies in the slab $(-\frac{\pi}{2},\frac{\pi}{2})\times \R^{2}$ and in no smaller slab. If the squash-down $\Omega_\ast$ of $\{\M_t\}_{t\in(-\infty,\,\omega)}$ is nondegenerate, then $\M_t$ is symmetric under reflection across $\{0\}\times\R^2$ for each $t\in(-\infty,\omega)$.
\end{theorem}

The proof proceeds by Alexandrov reflection. It is complicated by the lack of compactness (even when the timeslices are compact, since the time interval of existence is infinite). To deal with this, we exploit Theorem \ref{thm:circumscribed} using a parabolic version of the ``tilted hyperplane'' argument introduced by Korevaar, Kusner and Solomon in the context of constant mean curvature surfaces \cite{KorevaarKusnerSolomon}. 
\begin{proof}[Proof of Theorem \ref{thm:reflection symmetry}]

Let $\Omega_\ast$ be the squash-down of the solution, which we know is a circumscribed convex body. 
Set $Z=\Omega_\ast\cap S^1$ and consider points $z_i\in S^1$, $i=1,2,3$, as follows:
\begin{itemize}
\item[(i)] In the compact case, we take three points $\{z_1, z_2, z_3\}\in Z$ placed clockwise on $S^1$ so that the smaller angle between any two of them is less than $\pi$.
\item[(ii)] In the noncompact case, we take $\{z_1, z_2\}\in Z$ so that $Z$ is contained in the clockwise arc of $S^1$ between $z_1$ and $z_2$. The third point $z_3\in S^1$ is selected (in the clockwise arc between $z_2$ and $z_1$) so that the halfline $\ell=\{\rho z_3: \r\ge 0\}$ is contained in $\Omega_\ast$. Note that the smaller angle between any two of these three points  is less than $\pi$.
\end{itemize}
 Let $\pi_i$ be the halfplane spanned by $z_{i}$ and the $x$-axis, for $i=1, 2, 3$, and $W_i$ be the convex region of $\R^3$ bounded by $\pi_i$ and $\pi_{i+1}$, for $i=1, 2$, and by $\pi_3$ and $\pi_1$ for $i=3$.
We then have the following.


\begin{claim}\label{boundary reflection}
For any $\d>0$, there exists $t_\d$ such that
\begin{equation}\label{bdryrefl1}
|\a(p, s)|<\d\;\;\text{for all}\;\; p\in   \partial W_i\,,\,\, s\in(-\infty, t_\d)\,,\,\,\text{and}\,\, i=1,2, 3\,,
\end{equation}
where $\a(p,t)=\frac12(h_++h_-)$ is defined as in \eqref{afct} with respect to the mid-hyperplane $\pi=\{x_1=0\}$. 
\end{claim}
\begin{proof}[Proof of Claim \ref{boundary reflection}]
We first prove \eqref{bdryrefl1} on $\pi_i$ in case $\pi_i\cap \Omega_\ast$ is bounded. Suppose to the contrary that  there are $t_j\to-\infty$ and $p_j\in\pi_i$ such that
\begin{equation}\label{eq:refl translators contra}
\a(p_j, t_j)\ge \d\,,
\end{equation}
the  case $\a(p_j, t_j)\le -\d$ being treated similarly. By convexity, there exists a point $q_j\in \M_{t_j}\cap \pi_i$, such that the normal $\nu_j$ to $\M_{t_j}$ at $q_j$ satisfies $\langle \nu_j, e_1\rangle=0$. After passing to a subsequence, and by definition of the squash-down, we have that
\[
\nu_j\to z_i\,.
\]
By Proposition \ref{prop:converging sequences}, after passing to a further subsequence, the sequence of flows $\{\M^j_t\}_{t\in(-\infty,\,\omega-t_j)}$ defined by
\[
\M^j_s=\M_{s+t_j}-q_j\,,
\]
converges to a translator that satisfies $H(z_i, s)=\s_\ast(z_i)=1$ and thus is a Grim plane that lives in a slab of width $\pi$. Since this is reflection symmetric with respect to $\pi$ and the solution is convex, we obtain a contradiction to  \eqref{eq:refl translators contra}.

When $\pi_3\cap \Omega_\ast$ is unbounded,  \eqref{bdryrefl1}  follows by convexity and noncompactness (in particular $\nu\cdot e_1\ne 0$) along with the fact that the width of $\M_t$ at the origin converges to $\pi$ as $t\to -\infty$ (with the width defined as in the beginning of Section \ref{sec:squash-down}).
\end{proof}
Next we show that the estimate of Claim \ref{boundary reflection} can be extended in the interior of the convex regions $W_i$.
\begin{claim}\label{interior claim}
\begin{equation}\label{intrefl1}
|\a(p, s)|<\d\;\;\text{for all}\;\; p\in   W_i\,,\,\, s\in(-\infty, t_\d)\,,\,\,\text{and}\,\, i=1,2,3\,.
\end{equation}
\end{claim}
\begin{proof}
Let $v_i=z_i+ z_{i+1}$ for $i=1, 2$ and $v_3= z_3+z_1$.
Given $\varepsilon>0$, define the ``tilted" hyperplane
\[
\pi^i_\e=\pi+\{ r(v_i+\e e_1):r>0\}\,\,\text{for any}\,\,i=1,2,3\,.
\]
The plane $\pi^i_\e$ separates $\R^{3}$ into two halfspaces, $P_\e^+$ containing the positive $x$-axis, and $P_\e^-$ containing the negative one. Since $W_i\cap[-\frac\pi2, \frac\pi2]\cap  P_\e^+$ is compact and, in any compact region, $\M_t$ converges to the two parallel hyperplanes, $\{x_1=\pm \pi/2\}$, as $t\to -\infty$, we obtain, for any $\d>0$, a time $t_{\e,\d}$ such that
\[
\sup_{p\in W_i, s\in(-\infty, t_{\e, \d})}\a_\e(p, s)<\d\,,
\]
where $\a_\e(p,t)=\frac12(h_++h_-)$ as in \eqref{afct} with respect to $\pi^i_\e$.
Now the maximum principle and Claim \ref{boundary reflection} (which is true with $\a$ replaced by $\a_\e$) imply that
\[
\sup_{p\in W_i, s\in(-\infty, t_{\d})}\a_\e(p, s)<\d\,.
\]
Note that $t_\d$ depends on $\d$ but not on $\e>0$ and thus, letting $\e\to0$,
\[
\sup_{p\in W_i, s\in(-\infty, t_{\d})}\a(p, s)<\d\,.
\]
Repeating the argument with a negative tilt yields the claim.
\end{proof}
Note that Claim \ref{interior claim} implies that
\[
|\a(p, s)|<\d \;\;\text{for all}\;\; s\in (-\infty, t_\d)\,.
\]
If the solution is compact, the maximum principle implies that we can extend this estimate up to the final time $\omega$ and letting $\d\to 0$ we obtain the result.

For noncompact solutions we cannot obtain the result from the maximum principle, as the ``first point of contact" might be occuring ``at infinity". More precisely, we need to exclude the following phenomenon: There exists a sequence $t_j\to \infty$ and points $p_j\in \M_{t_j}$ such that
\begin{equation}\label{final}
|a(p_j, t_j)|>\d \;\;\text{for all}\;\; j\,.
\end{equation}
Let $q_j^\pm\in M_{t_j}$ be such that $q_j^\pm= p_j+ h_\pm(p_j, t_j)$. Let $z_j^\pm$ be the normals to $\M_{t_j}$ at the points $q_j^\pm$. By the maximum principle, we can assume that, after passing to a subsequence 
\[
z_j^\pm\to z^\pm \in G(\Omega^\ast)\setminus G_\ast\,.
\]
We will use Proposition \ref{prop:converging sequences nc} to obtain a contradiction to \eqref{final}.
Let $z_j^+= \cos\phi_jw_j+\sin\phi_j e_1$, for $w_j\in G_\ast$ and $\phi_j\in (-\pi/2, \pi/2)$ and assume first that $z^+\neq e_1$, so that $w_j\to w_\infty\in G(\Omega^\ast)\setminus G_\ast$. Let $q_j\in \M_{t_j}$ be such that $X(w_j, t_j)= q_j$. By Proposition \ref{prop:converging sequences nc}, after passing to a further subsequence, the sequence of flows $\{\M^j_t\}_{t\in(-\infty,\infty)}$ defined by
\[
\M^j_t=\M_{t+t_j}-q_j
\]
converge to a translator that satisfies $H(w_\infty, t)=\s^\ast(w_\infty)=1$ and thus is a Grim plane that lives in a slab of width $\pi$. Since this is reflection symmetric with respect to $\pi$ and by convexity (which in particuar takes care of the case $z^+=e_1$), we obtain a contradiction to  \eqref{final}.
\end{proof}

\bibliographystyle{acm}
\bibliography{../../bibliography}

\end{document}